\documentclass[11pt,reqno]{amsart}

\usepackage{a4wide}
\usepackage{amssymb,amsmath,amsthm,mathrsfs}
\usepackage{graphicx}
\usepackage{float}
\usepackage{color, colortbl}
\usepackage{enumerate}
\usepackage{upref}
\usepackage[bookmarks=false]{hyperref}

\usepackage[T2A,T1]{fontenc}
\DeclareSymbolFont{cyrillic}{T2A}{cmr}{m}{n}
\DeclareMathSymbol{\D}{\mathalpha}{cyrillic}{196}

\DeclareMathOperator{\essup}{ess\,sup}
\DeclareMathOperator{\esinf}{ess\,inf}
\DeclareMathOperator{\var}{var}

\theoremstyle{plain}
\newtheorem{theorem}{Theorem}[section]
\newtheorem{lemma}{Lemma}[section]
\newtheorem{corollaryP}[theorem]{Corollary}
\newtheorem{propositionP}[theorem]{Proposition}
\newtheorem{proposition}[lemma]{Proposition}

\theoremstyle{definition}
\newtheorem{definition}[lemma]{Definition}

\newtheorem*{condition}{Condition}

\theoremstyle{remark}
\newtheorem{remark}[lemma]{Remark}

\usepackage{enumitem}
\makeatletter
\def\namedlabel#1#2{\begingroup
   #2%
 \def\@currentlabel{#2}%
   \phantomsection\label{#1}\endgroup
}

\def\S{Section}
\def\ie{{\em i.e.,} }

\newfont\bbf{msbm10 at 12pt}
\def\eps{\varepsilon}
\def\R{{\mathbb R}}
\def\N{{\mathbb N}}

\def\B{{\mathcal B}}
\def\I{{\mathcal I}}

\DeclareMathOperator*{\essinf}{ess\;inf}

\def\R{\ensuremath{\mathbb R}}
\def\N{\ensuremath{\mathbb N}}

\def\I{\ensuremath{{\bf 1}}}
\def\e{\ensuremath{\text e}}

\def\S{\ensuremath{\mathcal S}}
\def\RR{\ensuremath{\mathcal R}}

\def\B{\ensuremath{\mathcal B}}

\def\M{\ensuremath{\mathcal M}}

\def\L{\ensuremath{\mathcal L}}

\def\p{\ensuremath{\mathbb P}}

\def\RR{\ensuremath{\mathcal R}}

\def\WW{\ensuremath{\mathscr W}}
\def\aa{\ensuremath{\mathscr A}}

\def\X{\mathcal{X}}
\def\Y{\mathcal{Y}}
\def\ie{{\em i.e.}, }

\def\E{\mathbb E}

\def\1{{\bf 1}}

\def\TF{\mathcal{T}}

\numberwithin{equation}{section}

\def\dist{\ensuremath{\mbox{dist}}}

\def\dist{\mbox{dist}}

\def\le{\leqslant}
\def\ge{\geqslant}

\begin{document}

\title[]{Point processes of non stationary sequences generated by sequential and random dynamical systems}

\author[A. C. M. Freitas]{Ana Cristina Moreira Freitas}
\address{Ana Cristina Moreira Freitas\\ Centro de Matem\'{a}tica \&
Faculdade de Economia da Universidade do Porto\\ Rua Dr. Roberto Frias \\
4200-464 Porto\\ Portugal} \email{\href{mailto:amoreira@fep.up.pt}{amoreira@fep.up.pt}}
\urladdr{\url{http://www.fep.up.pt/docentes/amoreira/}}

\author[J. M. Freitas]{Jorge Milhazes Freitas}
\address{Jorge Milhazes Freitas\\ Centro de Matem\'{a}tica \& Faculdade de Ci\^encias da Universidade do Porto\\ Rua do
Campo Alegre 687\\ 4169-007 Porto\\ Portugal}
\email{\href{mailto:jmfreita@fc.up.pt}{jmfreita@fc.up.pt}}
\urladdr{\url{http://www.fc.up.pt/pessoas/jmfreita/}}

\author[M. Magalh\~aes]{M\'ario Magalh\~aes}
\address{M\'ario Magalh\~aes\\ Centro de Matem\'{a}tica da Universidade do Porto\\ Rua do
Campo Alegre 687\\ 4169-007 Porto\\ Portugal} \email{\href{mailto:mdmagalhaes@fc.up.pt}{mdmagalhaes@fc.up.pt}}

\author[S. Vaienti]{Sandro Vaienti}
\address{Sandro Vaienti\\ Aix Marseille Universit\'e, Universit\'e de Toulon, CNRS, CPT, 13009 Marseille, France. }
\email{vaienti@cpt.univ-mrs.fr}
\urladdr{\url{http://www.cpt.univ-mrs.fr/~vaienti/}}

\thanks{MM was partially supported  by FCT grant SFRH/BPD/89474/2012, which is supported by the program POPH/FSE. ACMF and JMF were partially supported by FCT projects FAPESP/19805/2014 and PTDC/MAT-CAL/3884/2014, PTDC/MAT-PUR/28177/2017, with national funds. All authors would like to thank the support of  CMUP (UID/MAT/00144/2013), which is funded by FCT with national (MCTES) and European structural funds through the programs FEDER, under the partnership agreement PT2020. JMF would like to thank the University of Toulon for the appointment as ``Visiting Professor'' during the year 2018.}

\begin{abstract}
We give general sufficient conditions to prove the convergence of marked point processes that keep record of the occurrence of rare events and of their impact for non-autonomous dynamical systems. We apply the results to sequential dynamical systems associated to uniformly expanding maps and to random dynamical systems given by fibred Lasota Yorke maps. 
\end{abstract}


\maketitle

\section{Introduction}\label{sec:introduction}

The complexity of the orbital structure of chaotic systems brought special attention to the study of limiting laws of stochastic processes arising from such systems, since they borrow at least some probabilistic predictability to their erratic behaviour. 

The first step in this research direction is usually the construction of invariant physical measures, which provide an asymptotic spatial distribution of the orbits in the phase space and endow the stochastic processes dynamically generated with stationarity. Ergodicity then gives strong laws of large numbers. The mixing properties of the system restore asymptotic independence and, in this way, allow to mimic iid processes and prove limiting laws for the mean, such as: central limit theorems, large deviation principles, invariance principles, among others. However, in many occasions the exact formula for the invariant measure is not available and one has to rely on reference measures with respect to which these processes are not stationary anymore. Loosening stationarity leads to non-autonomous dynamical systems for which the study of limit theorems is just at the beginning. We mention the recent works \cite{AHNTV15, HNTV17,NTV18} and references therein.

While the limiting laws mentioned so far pertain to the mean or average behaviour of the system, in the recent years, the study of the extremal behaviour, ie, the laws that rule the appearance of abnormal observations along the orbits of the system has suffered an unprecedented development (\cite{LFFF16}). This study is deeply connected with the recurrence properties to certain regions of the phase space and was initially performed under stationarity. Very recently, in \cite{FFV17,FFV18}, the authors developed tools to obtain the limiting distribution for the partial maxima of non-stationary stochastic processes arising from sequential dynamical systems (\cite{BB84,CR07}) and random transformations or randomly perturbed systems (\cite{K86,K88}). In the case of random transformations, we also mention the papers \cite{RSV14,R14,RT15}, where limiting laws for the waiting time to hit/return to shrinking target sets in the phase space (which are related to the existence of limiting laws for the maximum \cite{C01,FFT10,FFT11}) were obtained for random dynamical systems. 

The main purpose of this paper is to enhance the study of rare events for non-autonomous systems and, therefore, in a non-stationary context, by considering the convergence of point processes instead of the more particular distributional limiting properties of the maximum or the hitting/return times statistics. Point processes have revealed as a powerful tool to study the extremal behaviour of stationary systems. The most simple point processes, the Rare Events Point Processes (REPP) keep track of the number of exceedances (abnormally high values) observed along the orbits of the system and allow to recover relevant information such as the expected time between the occurrence of extremal events, the intensity of clustering,  the distribution of the higher order statistics such as the maximum. For stationary systems, they were studied in \cite{FFT13}. We will also consider more sophisticated Marked Point Processes of Rare Events (which are random measures), studied for autonomous systems in \cite{FFM18} and which not only keep track of the number of exceedances but also of their impact. In the presence of clustering of rare events, we will be particularly interested in Area Over Threshold (AOT) marked point processes, which sum all the excesses over a certain threshold within a cluster, and Peak Over Threshold (POT) marked point point processes, which consider the record impact of the highest exceedance by taking the maximum excess within a cluster. The first allows to study the effect of aggregate damage, while the second focuses on the sensitivity to very high impacts. The potential of interest of these results is quite transversal, but we mention particularly the possible applications to climate dynamics where the study of extreme events for dynamical systems have proved to be very useful in the analysis of meteorological  data (see for example \cite{SKFG16,MCF17,MCBF18,FAMR19}).  

The paper is structured as follows. In Section~\ref{sec:setting}, we generalise the theory developed in \cite{FFV17} in order to obtain the convergence of Marked Point Processes of Rare Events (MREPP). In particular, we introduce the notation, concepts and conditions that allow us to state a result that establishes the convergence of the MREPP to a compound Poisson process for non-stationary stochastic processes, under some amenable conditions designed for application to non-autonomous systems. We believe that formula \eqref{eq:multiplicity} which gives the multiplicity distribution of the limiting compound Poisson process has an interest on its own. Section~\ref{sec:proof-convergence} is dedicated to the proof of the main convergence result stated in the previous section. In Section~\ref{sec:sequential}, we make a non-trivial application of our main convergence result to some sequential dynamical systems studied in \cite{CR07}, deriving exact formulas for the limiting multiplicity distribution. In Section~\ref{sec:random}, we establish a convergence limiting result of the MREPP  in the random dynamical systems setting, where we consider fibred LasotaYorke maps which were introduced in the recent paper  \cite{DFGV18}.

\section{The setting and statement of results}
\label{sec:setting}

Let $X_0, X_1, \ldots$ be a stochastic process\index{random/stochastic process}, where each r.v. $X_i:\Y\to\R\cup\{\pm\infty\}$ is defined on the measure space $(\Y,\mathcal B,\p)$. We assume that $\Y$ is a sequence space with a natural product structure so that each possible realisation of the stochastic process\index{random/stochastic process} corresponds to a unique element of $\Y$ and there exists a measurable map $\TF:\Y\to\Y$, the time evolution map, which can be seen as the passage of one unit of time, so that
\[
X_{i-1}\circ\TF=X_{i}, \quad \mbox{for all $i\in\N$}.
\]
The $\sigma$-algebra $\mathcal B$ can also be seen as a product $\sigma$-algebra adapted to the $X_i$'s.
For the purpose of this paper, $X_0, X_1,\ldots$ is possibly non-stationary. Stationarity would mean that $\p$ is $\TF$-invariant. Note that $X_i=X_0\circ\TF_i$, for all $i\in\N_0$, where $\TF_i$ denotes the $i$-fold composition of $\TF$, with the convention that $\TF_0$ denotes the identity map on $\Y$. In the applications below to sequential dynamical systems, we will have that $\TF_i=T_i\circ\ldots\circ T_1$ will be the concatenation of $i$ possibly different transformations $T_1,\ldots,T_i$, so that
\begin{equation}
\label{eq:def-stat-stoch-proc-DS}
X_n=\varphi\circ \TF_n,\quad \mbox{for all } n\in {\mathbb N}
\end{equation}
for some given observable $\varphi:\Y\to\R\cup\{\pm\infty\}$

Each random variable $X_i$ has a marginal distribution function (d.f.) denoted by $F_i$, \ie $F_i(x)=\p(X_i\leq x)$. Note that the $F_i$, with $i\in\N_0$, may all be distinct from each other. For a d.f. $F$ we let $\bar F=1-F$. We define $u_{F_i}=\sup\{x:F_i(x)<1\}$ and let $F_i(u_{F_i}-):=\lim_{h\to 0^+} F_i(u_{F_i}-h)=1$ for all $i$. We will consider the limiting law of
$$\mathbf{P}_{H,n}:=\p(X_0\leq u_{n,0},X_1\leq u_{n,1},\ldots,X_{Hn-1}\leq u_{n,Hn-1})$$
as $n\to\infty$, where $\{u_{n,i},i\leq Hn-1,n\geq 1\}$ is considered a real-valued boundary, with $H\in\N$.

We assume throughout the paper that
\begin{equation}
\bar F_{n,\max}(H):=\max\{\bar F_i(u_{n,i}), i\leq Hn-1\}\to 0, \quad \mbox{ as }n\to\infty,
\label{FmaxH}
\end{equation}

and, for some $\tau>0$,
\begin{equation}
\label{F_hn}
\sum_{i=0}^{h_n-1}\bar F_i(u_{n,i})=\frac{h_n}{n}\tau+o(1),
\end{equation}
for any unbounded increasing sequence of positive integers $h_n\leq Hn$. In particular, we have
\begin{equation}
F_{H,n}^*:=\sum_{i=0}^{Hn-1}\bar F_i(u_{n,i})\to H\tau, \quad \mbox{ as }n\to\infty.
\end{equation}

The most simple point processes that we will consider here keep track of the exceedances of the high thresholds $u_{n,i}$ by counting the number of such exceedances on a rescaled time interval. These thresholds are chosen such that
\begin{equation}
F^*_{1,n}=\sum_{i=0}^{n-1}\bar F_i(u_{n,i})\to\tau,\;\quad \mbox{ as }n\to\infty,
\label{F_n}
\end{equation}
so that the average number of exceedances among the first $n$ observations is kept, approximately, at the constant frequency $\tau>0$.

\subsection{Random measures and weak convergence}

We start by introducing the notions of \emph{random measures} and, in particular, \emph{point processes} and \emph{marked point processes}. One could introduce these concepts on general locally compact topological spaces with countable basis, but we will restrict to the case of the positive real line $[0,\infty)$ equipped with its Borel $\sigma$-algebra $\B_{[0,\infty)}$, where our applications lie. Consider a positive measure $\nu$ on $\B_{[0,\infty)}$. We say that $\nu$ is a Radon measure if $\nu(A)<\infty$, for every bounded set $A\in\B_{[0,\infty)}$. Let $\M:=\M([0,\infty))$ denote the space of all Radon measures defined on $([0,\infty),\B_{[0,\infty)})$. We equip this space with the vague topology, i.e., $\nu_n\to \nu$ in $\M([0,\infty))$ whenever $\int\psi\,d\nu_n\to \int\psi\,d\nu$ for every continuous function $\psi:[0,\infty)\to \R$ with compact support. Consider the subsets of $\M$ defined by $\M_p:=\{\nu\in\M: \nu(A)\in\N \mbox{ for all $A\in\B_{[0,\infty)}$}\}$ and $\M_a:=\{\nu\in\M: \nu \mbox{ is an atomic measure}\}$. A \emph{random measure} $M$ on $[0,\infty)$ is a random element of $\M$, \ie let $(\X,\B_\X, \p)$ be a probability space, then any measurable $M:\X\to \M$ is a random measure on $[0,\infty)$. A \emph{point process} $N$ and \emph{marked point process} $A$ are defined similarly as random elements on $\M_p$ and $\M_a$, respectively.

A point measure $\nu$ of $\M_p$ can be written as
$\nu=\sum_{i=1}^\infty \delta_{x_i}$,
where $x_1, x_2, \ldots$ is a collection of not necessarily distinct points in $[0,\infty)$ and $\delta_{x_i}$ is the Dirac measure at $x_i$, \ie for every $A\in\B_{[0,\infty)}$, we have that $\delta_{x_i}(A)=1$ if $x_i\in A$ and $\delta_{x_i}(A)=0$, otherwise. The elements $\nu$ of $\M_a$ can be written as
$\nu=\sum_{i=1}^\infty d_i \delta_{x_i}$,
where $x_1, x_2, \ldots\in [0,\infty)$ and $d_1, d_2,\ldots\in[0,\infty)$.

A concrete example of a marked point process, which in particular will appear as the limit of the marked point processes, is the following:
\begin{definition}
\label{def:compound-poisson-process}
Let $T_1, T_2,\ldots$ be an i.i.d. sequence of r.v. with common exponential distribution of mean $1/\theta$. Let $D_1, D_2, \ldots$ be another i.i.d. sequence of r.v., independent of the previous one, and with d.f. $\pi$. Given these sequences, for $J\in\B_{[0,\infty)}$, set
$$
A(J)=\int \I_J\;d\left(\sum_{i=1}^\infty D_i \delta_{T_1+\ldots+T_i}\right).
$$
Let $\X$ denote the space of all possible realisations of $T_1, T_2,\ldots$ and $D_1, D_2, \ldots$, equipped with a product $\sigma$-algebra and measure, then $A:\X\to \M_a([0,\infty))$ is a marked point process which we call a compound Poisson process of intensity $\theta$ and multiplicity d.f. $\pi$.
\end{definition}
\begin{remark}
\label{rem:poisson-process}
When $D_1, D_2, \ldots$ are integer valued positive random variables, $\pi$ is completely defined by the values $\pi_k=\p(D_1=k)$, for every $k\in\N_0$ and $A$ is actually a point process. If $\pi_1=1$ and $\theta=1$, then $A$ is the standard Poisson process and, for every $t>0$, the random variable $A([0,t))$ has a Poisson distribution of mean $t$.
\end{remark}

Now, we will give a definition of convergence of random measures (for more details, see \cite{K86}).
\begin{definition}
\label{def:convergence-point-processes}
Let $(M_n)_{n\in\N}:\X\to\M$ be a sequence of random measures defined on a probability space $(\X,\mathcal B_\X, \mu)$ and let $M:Y \to \M$ be another random measure defined on a possibly distinct probability space $(Y,\mathcal B_Y, \nu)$. We say that $M_n$ converges weakly to $M$ if, for every bounded continuous function $\varphi$ defined on $\M$, we have 
$$\lim_{n\to\infty}\int \varphi d \mu\circ M_n^{-1}=\int \varphi d \nu\circ M^{-1}.$$ 
We write $M_n \stackrel{\mu}{\Longrightarrow} M $.
\end{definition}

Checking the convergence of random measures using the definition is often quite hard, hence, it is useful to translate it into convergence in distribution of more tractable random variables or in terms of Laplace transforms. For that purpose, we let $\S$ denote the semi-ring of subsets of $\R_0^+$ whose elements
are intervals of the type $[a,b)$, for $a,b\in\R_0^+$. Let $\RR$
denote the ring generated by $\S$. Recall that for every $J\in\RR$
there are $\varsigma\in\N$ and $\varsigma$ disjoint intervals $I_1,\ldots,I_\varsigma\in\S$ such that
$J=\dot\cup_{i=1}^\varsigma I_j$. In order to fix notation, let
$a_j,b_j\in\R_0^+$ be such that $I_j=[a_j,b_j)\in\S$.

\begin{definition}
\label{def:Laplace-rv}
Let $Z$ be a non-negative, random variable with distribution $F$. For every $y\in\R_0^+$, the \emph{Laplace transform} $\phi(y)$ of the distribution $F$ is given by
\[
\phi(y):=\E\left(\e^{-yZ}\right)=\int \e^{-yZ} d\mu_F,
\]
where $\mu_F$ is the Lebesgue-Stieltjes probability measure associated to the distribution function $F$.
\end{definition}

\begin{definition}
\label{def:Laplace-point-process}
For a random measure $M$ on $\R_0^+$ and $\varsigma$ disjoint intervals $I_1, I_2,\ldots, I_\varsigma\in\S$ and non-negative $y_1, y_2,\ldots,y_\varsigma$, we define the \emph{joint Laplace transform} $\psi(y_1, y_2,\ldots,y_\varsigma)$ by
\[
\psi_M(y_1, y_2,\ldots,y_\varsigma)=\E\left(\e^{-\sum_{\ell=1}^\varsigma y_\ell M(I_\ell)}\right).
\]
\end{definition}

If $M$ is a compound Poisson point process with intensity $\lambda$ and multiplicity distribution $\pi$, then given $\varsigma$ disjoint intervals $I_1, I_2,\ldots, I_\varsigma\in\S$ and non-negative $y_1, y_2,\ldots,y_\varsigma$ we have:
\[
\psi_M(y_1, y_2,\ldots,y_\varsigma)=\e^{-\lambda\sum_{\ell=1}^\varsigma (1-\phi(y_\ell))|I_\ell|},
\]
where $\phi(y)$ is the Laplace transform of the multiplicity distribution $\pi$.

\begin{remark}
\label{rem:convergence-point-processes}
By \cite[Theorem~4.2]{K86}, the sequence of random measures $(M_n)_{n\in\N}$ converges weakly to the random measure $M$ iff the sequence of vector r.v. $(M_n(J_1), \ldots, M_n(J_\varsigma))$ converges in distribution to $(M(J_1), \ldots, M(J_\varsigma))$, for every $\varsigma\in\N$ and all disjoint $J_1,\ldots, J_\varsigma\in\S$ such that $M(\partial J_\ell)=0$ a.s., for $\ell=1,\ldots,\varsigma$, which will be the case if the respective joint Laplace transforms $\psi_{M_n}(y_1, y_2,\ldots,y_\varsigma)$ converge to the joint Laplace transform $\psi_M(y_1, y_2,\ldots,y_\varsigma)$, for all $y_1,\ldots, y_\varsigma\in[0,\infty)$.
\end{remark}

\subsection{Marked Point Processes of Rare Events}

Before we give the formal definition of Marked Point Processes of Rare Events, we need to introduce some notation and definitions that will also be useful to understand the conditions that we will introduce in order to prove their weak convergence.

Let $A\in\B$. We define a function that we refer to as \emph{first hitting time function} to $A$, denoted by $r_A:\X\to\N\cup\{+\infty\}$ where
\begin{equation}
\label{eq:hitting-time}
r_A(x)=\min\left\{j\in\N\cup\{+\infty\}:\; f^j(x)\in A\right\}.
\end{equation}
The restriction of $r_A$ to $A$ is called the \emph{first return time function} to $A$. We define the \emph{first return time} to $A$, which we denote by $R(A)$, as the essential infimum of the return time function to $A$,
\begin{equation}
\label{eq:first-return}
R(A)=\essinf_{x\in A} r_A(x).
\end{equation}

In what follows, for every $A\in\mathcal B$, we denote the complement of $A$ as $A^c:=\mathcal Y\setminus A$.

Let $\mathbb A:=(A_0,A_1,\ldots)$ be a sequence of events such that $A_i\in \TF_i^{-1}\mathcal B$. For some $s,\ell\in \N_0$, we define
\begin{equation}
\label{eq:W-def}
\mathscr W_{s,\ell}(\mathbb A)=\bigcap_{i=s}^{s+\ell-1} A_i^c,
\end{equation}
which forbids the occurrence of $A_i$ during the time interval between $s$ and $s+\ell-1$.

Given a set of thresholds $u_{n,i}$, for each $n$, $i$ and $j\in\N_0$ with $j<Hn-i$, we set
\begin{align*}
&U^{(0)}_{j,n,i}:=\{X_i>u_{n,i}\}\\
&Q^{(0)}_{j,n,i}:=U^{(0)}_{j,n,i}\cap\bigcap_{\ell=1}^j(U^{(0)}_{j,n,i+\ell})^c=\{X_i>u_{n,i},X_{i+1}\leq u_{n,i+1},...,X_{i+j}\leq u_{n,i+j}\}
\end{align*}

and define the following events, for $\kappa\in\N$:
\begin{align*}
&U^{(\kappa)}_{j,n,i}:=U^{(\kappa-1)}_{j,n,i}\setminus Q^{(\kappa-1)}_{j,n,i}=U^{(\kappa-1)}_{j,n,i}\cap\bigcup_{\ell=1}^j U^{(\kappa-1)}_{j,n,i+\ell}\\
&Q^{(\kappa)}_{j,n,i}:=U^{(\kappa)}_{j,n,i}\cap\bigcap_{\ell=1}^j(U^{(\kappa)}_{j,n,i+\ell})^c.
\end{align*}

If $j=0$ then $Q^{(0)}_{0,n,i}=U^{(0)}_{0,n,i}=\{X_i>u_{n,i}\}$ and $Q^{(\kappa)}_{0,n,i}=U^{(\kappa)}_{0,n,i}=\emptyset$ for $\kappa\in\N$.

For $j\geq Hn-i$, we set $Q^{(\kappa)}_{j,n,i}=U^{(\kappa)}_{j,n,i}=\emptyset$ for all $\kappa\in\N_0$.

Also, let $U^{(\infty)}_{j,n,i}:=\bigcap_{\kappa=0}^\infty U^{(\kappa)}_{j,n,i}$. Note that $Q^{(\kappa)}_{j,n,i}=U^{(\kappa)}_{j,n,i}\setminus U^{(\kappa+1)}_{j,n,i}$ for $\kappa\in\N_0$ and, therefore,
\[U^{(0)}_{j,n,i}=\bigcup_{\kappa=0}^{\infty}Q^{(\kappa)}_{j,n,i}\cup U^{(\infty)}_{j,n,i}.\]

\begin{remark}
\label{sizeofcluster}
The points in $U^{(\kappa)}_{j,n,i}$ are points whose orbit represents a cluster of size at least $\kappa+1$, since there will be points in each $U^{(\kappa')}_{j,n,i}$ with $k'$ taking values between $\kappa$ and $0$. On the other hand, points in $Q^{(\kappa)}_{j,n,i}$ are points whose orbit represents a cluster of size $\kappa+1$ exactly.
\end{remark}

For each $i\in\N_0$ and $n\in\N$, let $R_{j,n,i}:=\min\{r\in\N: Q_{j,n,i}^{(0)}\cap Q_{j,n,i+r}^{(0)}\neq\emptyset\}$.
We assume that there exists $q\in\N_0$ such that:
\begin{equation}
\label{def:q}
q=\min\left\{j\in\N_0:\lim_{n\to\infty} \min_{i\leq n}\left\{R_{j,n,i}\right\}=\infty\right\}.
\end{equation}

When $q=0$ then $Q_{0,n,i}^{(0)}$ corresponds to an exceedance of the threshold $u_{n,i}$ and we expect no clustering of exceedances.

When $q>0$, heuristically one can think that there exists an underlying periodic phenomenon creating short recurrence, \ie clustering of exceedances, when exceedances occur separated by at most $q$ units of time then they belong to the same cluster. Hence, the sets $Q_{q,n,i}^{(0)}$ correspond to the occurrence of exceedances that escape the periodic phenomenon and are not followed by another exceedance in the same cluster. We will refer to the occurrence of $Q_{q,n,i}^{(0)}$ as the occurrence of an escape at time $i$, whenever $q>0$.

Given an interval $I\in\S$, $x\in\X$ and $u_{n,i}\in\R$, we define
\[N_{n,I}(x):=\sum_{i\in I\cap\N_0} \I_{Q^{(0)}_{q,n,i}}(x).\]
Let $i_1(x)<i_2(x)<\ldots<i_{N_{n,I}(x)}(x)$ denote the times at which the orbit of $x$ entered $Q^{(0)}_{q,n,i}$. We now define the cluster periods: for every $k=1,\ldots,N_{n,I}(x)-1$ let $I_k(x)=[i_k(x),i_{k+1}(x))$ and set $I_0(x)=[\min I,i_1(x))$ and $I_{N_{n,I}(x)}(x)=[i_{N_{n,I}(x)}(x),\sup I)$.

In order to define the marks for each cluster we consider the following mark functions that depend on the levels $u_{n,i}$ and on the random variables in a certain time frame $I\in\S$:
\begin{equation}
\label{eq:mark-type}m_n(I):=\begin{cases}
\sum_{i\in I\cap\N_0}(X_i-u_{n,i})_+ & \text{AOT case}\\
\max_{i\in I\cap\N_0}\{(X_i-u_{n,i})_+\} & \text{POT case}\\
\sum_{i\in I\cap\N_0}\I_{X_i>u_{n,i}} & \text{REPP case},
\end{cases}
\end{equation}
where $(y)_+=\max\{y,0\}$ and
when $I\cap\N_0\neq\emptyset$. Also set $m_n(I):=0$ when $I\cap\N_0=\emptyset$.

Finally, we set
\[\aa_n(I)(x):=\sum_{k=0}^{N_{n,I}(x)} m_n(I_k(x)).\]

In order to avoid degeneracy problems in the definition of the marked point processes we need to rescale time by the factor
\[v_n:=n/F^*_{1,n}\]
so that the expected average number of exceedances of the levels $u_{n,i}$ for $i=0,\ldots,n$ in each time frame considered is kept `constant' as $n\to\infty$. Recall that the levels $u_{n,i}$ satisfy \ref{F_n}, and therefore $v_n\sim\frac{n}{\tau}$, where we use the notation $A(n)\sim B(n)$, when $\lim_{n\to \infty} \frac{A(n)}{B(n)}=1$. Hence, we introduce the following notation. For $I=[a,b)\in\S$ and $\alpha\in \R$, we denote $\alpha I:=[\alpha a,\alpha b)$ and $I+\alpha:=[a+\alpha,b+\alpha)$. Similarly, for $J\in\RR$, such that $J=J_1\dot\cup\ldots\dot\cup J_k$, define $\alpha J:=\alpha J_1\dot\cup\cdots\dot\cup \alpha J_k$ and $J+\alpha:=(J_1+\alpha)\dot\cup\cdots\dot\cup (J_k+\alpha)$.

\begin{definition}
\label{def:MREPP}
We define the \emph{marked rare event point process} (MREPP) by setting for every $J\in\RR$, with $J=J_1\dot\cup\ldots\dot\cup J_k$, where $J_i\in \S$ for all $i=1,\ldots, k$,
\begin{equation}
\label{eq:def-MREPP}
A_n(J):=\sum_{i=1}^k\aa_n(v_n J_i).
\end{equation}
\end{definition}

When $m_n$ given in \eqref{eq:mark-type} is as in the AOT case, then the MREPP $A$ computes the sum of all excesses over the threshold $u_n$ and, in such case, we will refer to $A$ as being an \emph{area over threshold} or AOT MREPP. Observe that in this case we may write:
\[A_n(J)=\sum_{i\in v_n J\cap\N_0} (X_i-u_{n,i})_+.\]

When $m_n$ given in \eqref{eq:mark-type} is as in the POT case, then the MREPP $A_n$ computes the sum of the largest excess (peak) over the threshold $u_{n,i}$ within each cluster and, in such case, we will refer to $A_n$ as being a \emph{peaks over threshold} or POT MREPP.

When $m_n$ given in \eqref{eq:mark-type} is as in the REPP case, then the MREPP $A_n$ is actually a point process that counts the number of exceedances of $u_{n,i}$ and, in such case, we will refer to $A_n$ as being a \emph{rare events point process} or REPP, as it was referred in \cite{FFT13}. Observe that in this case we have:
\[A_n(J)=\sum_{i\in v_n J\cap\N_0} \I_{X_i>u_{n,i}}.\]

If $q=0$ then the AOT MREPP and the POT MREPP coincide and both compute the sum of all excesses over the threshold $u_{n,i}$. In such situation we say that $A_n$ is an \emph{excesses over threshold} (EOT) MREPP.

Next, we will introduce the dependence conditions $\D_q(u_{n,i})^*$ and $\D'_q(u_{n,i})^*$, which are the analogous of conditions $\D_p(u_n)$ and $\D'_p(u_n)$ considered in \cite{FFT15}, but designed to establish the convergence of MREPP (AOT, POT or REPP), which allow us to state our main result. Before we do that, we need to introduce some additional notation and definitions.

For $x\geq 0$ and $\kappa\in\N_0$, we define the following events:
\begin{align*}
R^{(\kappa)}_{n,i}(x)&:=Q^{(\kappa)}_{q,n,i}\cap\left\{m_n(I_{\kappa})>x\right\}\\
B_{n,i}(x)&:=\bigcup_{\kappa=0}^{\infty}R^{(\kappa)}_{n,i}(x)\cup U^{(\infty)}_{q,n,i}\\
A_{n,i}(x)&:=B_{n,i}(x)\cap\bigcap_{\ell=1}^q (B_{n,i+\ell}(x))^c
\end{align*}
where $I_{\kappa}=[i,i_{\kappa,\kappa}+1)$ and $i_{\kappa,j}$ denotes the times at which the orbit of the considered point entered $Q^{(\kappa-j)}_{q,n,i}$, with $i=i_{\kappa,0}<i_{\kappa,1}<i_{\kappa,2}<\ldots<i_{\kappa,j}<\ldots<i_{\kappa,\kappa}$ (see Remark~\ref{sizeofcluster}).

In particular, for $x=0$ we have
\begin{align*}
R^{(\kappa)}_{n,i}(0)&=Q^{(\kappa)}_{q,n,i}\\
B_{n,i}(0)&=\bigcup_{\kappa=0}^{\infty}Q^{(\kappa)}_{q,n,i}\cup U^{(\infty)}_{q,n,i}=U^{(0)}_{q,n,i}\\
A_{n,i}(0)&=U^{(0)}_{q,n,i}\cap\bigcap_{\ell=1}^q (U^{(0)}_{q,n,i+\ell})^c=Q^{(0)}_{q,n,i}
\end{align*}
and, if $q=0$,
\begin{align*}
R^{(0)}_{n,i}(x)&=\{X_i>u_{n,i},m_n([i,i+1))>x\} \mbox{, } \quad R^{(\kappa)}_{n,i}(x)=\emptyset \mbox{ for }\kappa\in\N\\
A_{n,i}(x)&=B_{n,i}(x)=R^{(0)}_{n,i}(x).\\
\end{align*}

\begin{condition}[$\D_q(u_{n,i})^*$]\label{cond:Dq*}
We say that $\D_q(u_{n,i})^*$ holds for the sequence $X_0,X_1,X_2,\ldots$ if for $t,n\in\N$, $i=0,\ldots,Hn-1$, for $x_1,\ldots,x_{\varsigma}\geq 0$ and any $J=\cup_{i=2}^\varsigma I_j\in \mathcal R$ with $\inf\{x:x\in J\}\ge i+t$,
\[\left|\p\left(A_{n,i}(x_1)\cap \bigcap_{j=2}^\varsigma\{\aa_n(I_j)\leq x_j\}\right)-\p\left(A_{n,i}(x_1)\right)\p\left(\bigcap_{j=2}^\varsigma\{\aa_n(I_j)\leq x_j\}\right)\right|\leq \gamma_i(n,t),\]
where for each $n$ and each $i$ we have that $\gamma_i(n,t)$ is nonincreasing in $t$ and there exists a sequence $t_n^*=o(n)$ such that $t_n^*\bar F_{n,\max}(H)\to 0$ and $n\gamma_i(n,t_n^*)\to 0$ when $n\rightarrow\infty$.
\end{condition}

Note that the main advantage of this mixing condition when compared with condition $\Delta(u_n)$ used by Leadbetter in \cite{L91} or any other similar such condition available in the literature is that it follows easily from sufficiently fast decay of correlations and therefore is particularly useful when applied to stochastic processes arising from dynamical systems.

For $q\in\N_0$ given by \eqref{def:q}, consider the sequence $(t_n^*)_{n\in\N}$, given by condition $\D_q(u_{n,i})^*$ and let $(k_n)_{n\in\N}$ be another sequence of integers such that
\begin{equation}
\label{eq:kn-sequence}
k_n\to\infty\quad \mbox{and}\quad k_n t_n^* \bar F_{n,\max}(H)\to 0
\end{equation}
as $n\rightarrow\infty$ for every $H\in\N$.

Let us give a brief description of the blocking argument and postpone the precise construction of the blocks to Section~\ref{subsec:blocks}. We split the data into $k_n$ blocks separated by time gaps of size larger than $t_n^*$, where we simply disregard the observations in the corresponding time frame. In the stationary case, the blocks have the same size and the expected number of exceedances within each block is $\sim \tau/k_n$. Here, the blocks may have different sizes, denoted by $\ell_{H,n,1}, \ldots, \ell_{H,n,k_n}$, but, as in \cite{FFV17}, these are chosen so that the expected number of exceedances is again $\sim\tau/k_n$. Also, for $i=1,\ldots,k_n$, let $\L_{H,n,i}=\sum_{j=1}^i\ell_{H,n,j}$ and $\L_{H,n,0}=0$.

Note that gaps need to be big enough so that they are larger than $t_n^*$ but they also need to be  sufficiently small so that the information disregarded does not compromise the computations. This is achieved by choosing the number of blocks, which correspond to the sequence $(k_n)_{n\in\N}$, diverging but slowly enough so that the weight of the gaps is negligible when compared to that of the true blocks.

In order to guarantee the existence of a distributional limit one needs to impose some restrictions on the speed of recurrence.

\begin{condition}[$\D'_q(u_{n,i})^*$]\label{cond:D'q}
We say that $\D'_q(u_{n,i})^*$ holds for the sequence $X_0,X_1,X_2,\ldots$ if there exists a sequence $(k_n)_{n\in\N}$ satisfying \eqref{eq:kn-sequence} and such that, for every $H\in\N$,
\begin{equation}
\label{eq:D'rho-un}
\lim_{n\rightarrow\infty}\sum_{i=1}^{k_n} \sum_{j=\L_{H,n,i-1}}^{\L_{H,n,i}-1} \sum_{r>j}^{\L_{H,n,i}-1}\p\left(Q^{(0)}_{q,n,j}\cap\{X_r>u_{n,r}\}\right)=0
\end{equation}
\begin{equation}
\mbox{and}\quad\lim_{n\rightarrow\infty}\sum_{j=\L_{H,n,k_n}}^{Hn-1} \sum_{r>j}^{Hn-1}\p\left(Q^{(0)}_{q,n,j}\cap\{X_r>u_{n,r}\}\right)=0
\end{equation}
\end{condition}

Condition $\D'_q(u_{n,i})^*$ precludes the occurrence of clustering of escapes (or exceedances, when $q=0$).

\begin{remark}
Note that condition $\D'_q(u_{n,i})^*$ is an adjustment of a similar condition $\D'_p(u_{n})$ in \cite{FFT15} in the stationary setting, which is similar to condition $D^{(p+1)}(u_n)$ in the formulation of \cite[Equation (1.2)]{CHM91}, although slightly weaker.
\end{remark}

When $q=0$, observe that $\D'_q(u_{n,i})^*$ is very similar to $D'(u_{n,i})$ from H\"usler, which prevents clustering of exceedances, just as $D'(u_n)$ introduced by Leadbetter did in the stationary setting.

When $q>0$, we have clustering of exceedances, \ie the exceedances have a tendency to appear aggregated in groups (called clusters). One of the main ideas in \cite{FFT12} that we use here is that the events $Q_{q,n,i}^{(0)}$ play a key role in determining the limiting EVL and in identifying the clusters. In fact, when $\D'_q(u_{n,i})^*$ holds we have that every cluster ends with an entrance in $Q_{q,n,i}^{(0)}$, which means that the inter cluster exceedances must appear separated at most by $q$ units of time.

In this approach, condition $\D'_q(u_{n,i})^*$ plays a prominent role. In particular, note that if condition $\D'_q(u_{n,i})^*$ holds for some particular $q=q_0\in\N_0$, then it holds for all $q\geq q_0$, and so \eqref{def:q} is indeed the natural candidate to try to show the validity of $\D'_q(u_{n,i})^*$.

Now, we give a way of defining the Extremal Index (EI) using the sets $Q_{q,n,i}^{(0)}$. For $q\in\N_0$ given by \eqref{def:q}, we also assume that there exists $0\leq\theta\leq1$, which will be referred to as the EI, such that
\begin{equation}
\label{eq:EIH}
\lim_{n\to\infty}\max_{i=1,\ldots,k_n}\left\{\left|\theta\sum_{j=\L_{H,n,i-1}}^{\L_{H,n,i}-1}\bar F(u_{n,j})-\sum_{j=\L_{H,n,i-1}}^{\L_{H,n,i}-1}\p\left(Q_{q,n,j}^{(0)}\right)\right|\right\}=0.
\end{equation}
Moreover, we assume the existence of normalising factors $a_{n,j}$ for every $j=0,1,\ldots,Hn-1$ and $n\in\N$, and a probability distribution $\pi$ such that, for every $H\in\N$ and $x\geq 0$,
\begin{equation}
\label{eq:multiplicity}
\lim_{n\to\infty}\max_{j=0,1,\ldots,Hn-1}\left\{\left|\frac{\p(A_{n,j}(x/a_{n,j}))}{\p\left(Q^{(0)}_{q,n,j}\right)}-(1-\pi(x))\right|\right\}=0
\end{equation}
and in this way obtain a formula to compute the multiplicity distribution of the limiting compound Poisson process.

Finally, assuming that both $\D_q(u_{n,i})^*$ and $\D'_q(u_{n,i})^*$ hold, we give a technical condition which imposes a sufficiently fast decay of the probability of having very long clusters. We will call it $U\!LC_q(u_{n,i})$ that stands for `Unlikely Long Clusters'. Of course this condition is trivially satisfied when there is no clustering.

\begin{condition}[$U\!LC_q(u_{n,i})$]
We say that condition $U\!LC_q(u_{n,i})$ holds if, for all $H\in\N$ and $y>0$,
\[\lim_{n\to\infty}\;\sum_{i=1}^{k_n}\int_0^{\infty}\e^{-y x}\delta_{n,\L_{H,n,i-1},\ell_{H,n,i}}(x/a_n)dx=0,\]
\[\lim_{n\to\infty}\;\int_0^{\infty}\e^{-x}\delta_{n,\L_{H,n,k_n},Hn-\L_{H,n,k_n}}(x/a_n)dx=0,\]
\[\mbox{and}\quad\lim_{n\to\infty}\;\sum_{i=1}^{k_n}\int_0^{\infty}\e^{-y x}\delta_{n,\L_{H,n,i-1},\ell_{H,n,i}-t_{H,n,i}}(x/a_n)dx=0\]

where $a_n$ is such that $\p(A_{n,j}(x/a_{n,j}))=\p(A_{n,j}(x/a_n))$ for $a_{n,j}$ is as in \eqref{eq:multiplicity}, $\delta_{n,s,\ell}(x):=0$ for $q=0$ and, for $q>0$,
\begin{equation}
\label{eq:delta-definition}
\delta_{n,s,\ell}(x):=\sum_{\kappa=1}^{\lfloor\ell/q\rfloor}\sum_{j=s+\ell-\kappa q}^{s+\ell-1}\p\left(R^{(\kappa)}_{n,j}(x)\right)+\sum_{j=s}^{s+\ell-1}\sum_{\kappa>\lfloor\ell/q\rfloor}\p\left(R^{(\kappa)}_{n,j}(x)\right)+\sum_{j=1}^q \p(B_{n,s+\ell-j}(x))\end{equation}
is an integrable function in $\R^+$ for $n$ sufficiently large.
\end{condition}

Note that, by definition, condition $U\!LC_0(u_{n,i})$ always holds. Note also that $\delta_{n,s,\ell}(x)\leq\delta_{n,s',\ell'}(x)$ if $s+\ell=s'+\ell'$ and $\ell\leq\ell'$. In particular, if $U\!LC_q(u_{n,i})$ holds then, for all $H\in\N$ and $y>0$,
\[\lim_{n\to\infty}\;\sum_{i=1}^{k_n}\int_0^{\infty}\e^{-yx}\delta_{n,\L_{H,n,i}- t_{H,n,i},t_{H,n,i}}(x/a_n)dx=0\]

We are now ready to state the main convergence result:

\begin{theorem}
\label{thm:convergence}
Let $X_0,X_1,\ldots$ be given by \eqref{eq:def-stat-stoch-proc-DS} and $u_{n,i}$ be real-valued boundaries satisfying \eqref{FmaxH} and \eqref{F_hn}. Assume that $\D_q(u_{n,i})^*$, $\D'_q(u_{n,i})^*$ and $U\!LC_q(u_{n,i})^*$ hold, for some $q\in\N_0$. Assume the existence of $\theta$ satisfying \eqref{eq:EIH} and a normalising sequence $(a_n)_{n\in\N}$ such that $\p(A_{n,j}(x/a_{n,j}))=\p(A_{n,j}(x/a_n))$ for any $j=0,1,\ldots,Hn-1$, where $a_{n,j}$ are normalising factors such that \eqref{eq:multiplicity} holds for some probability distribution $\pi$. Then, the MREPP $a_n A_n$, where $A_n$ is given by Definition~\ref{def:MREPP} for any of the 3 mark functions considered in \eqref{eq:mark-type}, converges in distribution to a compound Poisson process $A$ with intensity $\theta$ and multiplicity d.f. $\pi$.
\end{theorem}

\begin{remark}
If the normalising factors $a_{n,j}$ don't depend on $j$, then we can naturally choose $a_n=a_{n,j}$ for every $n\in\N$.
\end{remark}

\begin{remark}
\label{rem:general-marks}
What is essential, about the mark function $m_u$ considered in \eqref{eq:mark-type} to define the respective MREPP, is that it satisfies the following assumptions:
\begin{enumerate}
\item $m_n(I)\geq 0$ and $m_n(\emptyset)=0$
\item $m_n(I)\leq m_n(J)$ if $I\subset J$
\item $m_n(I)=m_n(J)$ if $X_i\leq u_{n,i},\forall i\in (I\setminus J)\cap\N_0$
\end{enumerate}
Note that, in particular, we must have $m_n(I)=0 \mbox{ if } X_i\leq u_{n,i},\forall i\in I\cap\N_0$.

As long as the above assumptions hold then the conclusion of Theorem~\ref{thm:convergence} holds for the MREPP defined from such a mark function $m_n$ satisfying the three assumptions just enumerated.
\end{remark}

\section{Convergence of marked rare events point processes}
\label{sec:proof-convergence}

This section is dedicated to the proof of Theorem~\ref{thm:convergence}, whose argument follows the same thread as the one in the proof of \cite[Theorem~2.A.]{FFM18}

\subsection{The construction of the blocks}
\label{subsec:blocks}

The construction of the blocks is designed so that the expected number of exceedances in each block is the same. We follow closely the construction in \cite{FFV17}, which was inspired in \cite{H83,H86}.

For each $H,n\in\N$ we split the random variables $X_0, \ldots, X_{Hn-1}$ into $k_n$ initial blocks, where $k_n$ is given by \eqref{eq:kn-sequence}, of sizes $\ell_{H,n,1}, \ldots, \ell_{H,n,k_n}$ defined in the following way. Let as before $\L_{H,n,i}=\sum_{j=1}^i \ell_{H,n,j}$ and $\L_{H,n,0}=0$. Assume that $\ell_{H,n,1}, \ldots, \ell_{H,n,i-1}$ are already defined. Take $\ell_{H,n,i}$ to be the largest integer such that
\[\sum_{j=\L_{H,n,i-1}}^{\L_{H,n,i-1}+\ell_{H,n,i}-1}\bar F(u_{n,j})\leq \frac{F_{H,n}^*}{k_n}.\]

The final working blocks are obtained by disregarding the last observations of each initial block, which will create a time gap between each final block. The size of the time gaps must be balanced in order to have at least a size $t_n^*$ but such that its weight on the average number of exceedances is negligible when compared to that of the final blocks. For that purpose we define
$$\eps(H,n):=(t_n^*+1)\bar F_{n,\max}(H)\frac{k_n}{F_{H,n}^*}.$$
Note that by \eqref{F_hn} and \eqref{eq:kn-sequence}, it follows immediately that $\lim_{n\to\infty}\eps(H,n)=0$. Now, for each $i=1,\ldots,k_n$ let $t_{H,n,i}$ be the largest integer such that
$$\sum_{j=\L_{H,n,i}-t_{H,n,i}}^{\L_{H,n,i}-1}\bar F(u_{n,j})\leq \eps(H,n)\frac{F_{H,n}^*}{k_n}.$$

Hence, the final working blocks correspond to the observations within the time frame $\L_{H,n,i-1},\ldots,\L_{H,n,i}- t_{H,n,i}-1$, while the time gaps correspond to the observations in the time frame $\L_{H,n,i}-t_{H,n,i},\ldots,\L_{H,n,i}-1$, for all $i=1,\ldots,k_n$.

Note that $t_n^*<t_{H,n,i}<\ell_{H,n,i}$, for each $i=1,\ldots,k_n$. The second inequality is trivial. For the first inequality note that by definition of $t_{H,n,i}$ we have
$$\eps(H,n)\frac{F_{H,n}^*}{k_n}<\sum_{j=\L_{H,n,i}-t_{H,n,i}-1}^{\L_{H,n,i}-1}\bar F(u_{n,j})\leq (t_{H,n,i}+1)\bar F_{n,\max}(H).$$ The first inequality follows easily now by definition of $\eps(H,n)$.

Also, note that, by choice of $\ell_{H,n,i}$ we have
\[\frac{F_{H,n}^*}{k_n}\leq \sum_{j=\L_{H,n,i-1}}^{\L_{H,n,i}-1}\bar F(u_{n,j})+\bar F(u_{n,\L_{H,n,i}})\leq \sum_{j=\L_{H,n,i-1}}^{\L_{H,n,i}-1}\bar F(u_{n,j})+\bar F_{n,\max}(H)\]
and then it follows that
\begin{equation}
\label{eq:block-estimate}
\frac{F_{H,n}^*}{k_n}-\bar F_{n,\max}(H)\leq\sum_{j=\L_{H,n,i-1}}^{\L_{H,n,i}-1}\bar F(u_{n,j})\leq \frac{F_{H,n}^*}{k_n}.
\end{equation}

From the first inequality we get
\[F_{H,n}^*-k_n\bar F_{n,\max}(H)\leq\sum_{i=1}^{k_n}\sum_{j=\L_{H,n,i-1}}^{\L_{H,n,i}-1}\bar F(u_{n,j})\]
which implies that
\begin{equation}
\label{eq:block-estimate-last}
\sum_{j=\L_{H,n,k_n}}^{Hn-1}\bar F(u_{n,j})=F_{H,n}^*-\sum_{i=1}^{k_n}\sum_{j=\L_{H,n,i-1}}^{\L_{H,n,i}-1}\bar F(u_{n,j})\leq k_n\bar F_{n,\max}(H)
\end{equation}
which goes to $0$ as $n\to\infty$ by \eqref{eq:kn-sequence}.

\begin{proposition}
\label{prop:relation-balls-annuli-general}
Given events $B_0,B_1,\ldots\in\mathcal B$, let $r,q,s,\ell\in\N$ be such that $q<n$ and define $\mathbb B=(B_0,B_1,\ldots)$, $A_r=B_r\setminus \bigcup_{j=1}^{q} B_{r+j}$ and $\mathbb A=(A_0,A_1,\ldots)$. Then
$$
\left|\p(\mathscr W_{s,\ell}(\mathbb B))-\p(\mathscr W_{s,\ell}(\mathbb A))\right|\leq \sum_{j=1}^{q} \p\left(\mathscr W_{s,\ell}(\mathbb A)\cap (B_{s+\ell-j}\setminus A_{s+\ell-j})\right).
$$
\end{proposition}

\begin{proof}
Since $A_r\subset B_r$, then clearly $\mathscr W_{s,\ell}(\mathbb B)\subset \mathscr W_{s,\ell} (\mathbb A)$. Hence, we have to estimate the probability of $\mathscr W_{s,\ell}(\mathbb A)\setminus \mathscr W_{s,\ell}(\mathbb B)$.

Let $x\in\mathscr W_{s,\ell}(\mathbb A)\setminus \mathscr W_{s,\ell}(\mathbb B)$. We will see that there exists $j\in\{1,\ldots,q\}$ such that $x\in B_{s+\ell-j}$. In fact, suppose that no such $j$ exists. Then let $\kappa=\max\{i\in\{s,\ldots,s+\ell-1\}:\, x\in B_i\}$. Then, clearly, $\kappa<s+\ell-q$. Hence, if $x\notin B_j$, for all $i=\kappa+1,\ldots,s+\ell-1$, then we must have that $x\in A_\kappa$ by definition of $A$. But this contradicts the fact that $x\in\mathscr W_{s,\ell}(\mathbb A)$. Consequently, we have that there exists $j\in\{1,\ldots,q\}$ such that $x\in B_{s+\ell-j}$ and since $x\in\mathscr W_{s,\ell}(\mathbb A)$ then we can actually write $x\in B_{s+\ell-j}\setminus A_{s+\ell-j}$.

This means that $\mathscr W_{s,\ell}(\mathbb A)\setminus \mathscr W_{s,\ell}(\mathbb B)\subset \bigcup_{j=1}^q (B_{s+\ell-j}\setminus A_{s+\ell-j})\cap \mathscr W_{s,\ell}(\mathbb A)$ and then
\begin{multline*}
\big|\p(\mathscr W_{s,\ell}(\mathbb B))-\p(\mathscr W_{s,\ell}(\mathbb A))\big|=\p(\mathscr W_{s,\ell}(\mathbb A)\setminus \mathscr W_{s,\ell}(\mathbb B))\\
\leq \p\left(\bigcup_{j=1}^q (B_{s+\ell-j}\setminus A_{s+\ell-j})\cap \mathscr W_{s,\ell}(\mathbb A)\right)\leq\sum_{j=1}^{q} \p\left(\mathscr W_{s,\ell}(\mathbb A)\cap (B_{s+\ell-j}\setminus A_{s+\ell-j})\right),
\end{multline*}
as required.
\end{proof}

Applying this proposition to $B_i=B_{n,i}(x)$, we have the following lemma, which says that the probability of not entering $B_{n,i}(x)$ can be approximated by the probability of not entering $A_{n,i}(x)$ during the same period of time.
\begin{lemma}
\label{Lem:disc-ring}
For any $s,\ell\in\N$ and $x\geq 0$ we have
\[\left|\p\big(\WW_{s,\ell}(B_{n,i}(x))\big)-\p\big(\WW_{s,\ell}(A_{n,i}(x))\big)\right|\leq\sum_{i=1}^q \p(B_{n,s+\ell-i}(x))\]
\end{lemma}

Next we give an approximation for the probability of not entering $A_{n,i}(x)$ during a certain period of time.

\begin{lemma}
\label{lem:no-entrances-ring}
For any $s,\ell\in\N$ and $x\geq 0$ we have
\[\left|\p(\WW_{s,\ell}(A_{n,i}(x)))-\left(1-\sum_{i=s}^{s+\ell-1}\p(A_{n,i}(x))\right)\right|\leq \sum_{j=s}^{s+\ell-1}\sum_{r=j+1}^{s+\ell-1}\p\left(Q^{(0)}_{q,n,j}\cap\{X_r>u_{n,r}\}\right)\]
\end{lemma}
\begin{proof}
Since $(\WW_{s,\ell}(A_{n,i}(x)))^c=\cup_{i=s}^{s+\ell-1} A_{n,i}(x)$ it is clear that
\[\left|1-\p(\WW_{s,\ell}(A_{n,i}(x)))-\sum_{i=s}^{s+\ell-1}\p(A_{n,i}(x))\right|\leq \sum_{j=s}^{s+\ell-1}\sum_{r=j+1}^{s+\ell-1}\p(A_{n,j}(x)\cap A_{n,r}(x))\]

If $q>0$, the result follows by the fact that $A_{n,r}(x)\subset \{X_r>u_{n,r}\}$ and the fact that the occurrence of both $A_{n,j}(x)$ and $A_{n,r}(x)$ implies an escape, \ie the occurrence of $Q^{(0)}_{q,n,j_1}$ for some $j\leq j_1<r$ (otherwise, the occurrence of $A_{n,r}(x)$ and therefore of $B_{n,r}(x)$ would imply the occurrence of $B_{n,r_1}(x)$ for some $j+1\leq r_1\leq j+q$ which would contradict the occurrence of $A_{n,j}(x)$).

If $q=0$, the result follows immediately since $A_{n,i}(x)\subset \{X_i>u_{n,i}\}=Q^{(0)}_{0,n,i}$.

\end{proof}

The next lemma gives an error bound for the approximation of the probability of the process $\aa_n([s,s+\ell))$ not exceeding $x$ by the probability of not entering in $B_{n,i}(x)$ during the period $[s,s+\ell)$. In what follows, we use the notation $\aa_{n,s}^{s+\ell}:=\aa_n([s,s+\ell))$.

\begin{lemma}
\label{lem:entrances-ball-depth}
For any $s,\ell\in\N$ and $x\geq 0$ we have
\begin{align*}
\left|\p(\aa_{n,s}^{s+\ell}\leq x)-\p(\WW_{s,\ell}(B_{n,i}(x)))\right|&\leq \sum_{j=s}^{s+\ell-1}\sum_{r=j+1}^{s+\ell-1}\p\left(Q^{(0)}_{q,n,j}\cap \{X_r>u_{n,r}\}\right)\\
&+\sum_{\kappa=1}^{\lfloor\ell/q\rfloor}\sum_{i=s+\ell-\kappa q}^{s+\ell-1}\p\left(R^{(\kappa)}_{n,i}(x)\right)+\sum_{i=s}^{s+\ell-1}\sum_{\kappa>\lfloor\ell/q\rfloor}\p\left(R^{(\kappa)}_{n,i}(x)\right)
\end{align*}
if $q>0$, and in case $q=0$ we have
\[\left|\p(\aa_{n,s}^{s+\ell}\leq x)-\p(\WW_{s,\ell}(B_{n,i}(x)))\right|\leq \sum_{j=s}^{s+\ell-1}\sum_{r=j+1}^{s+\ell-1}\p(X_j>u_{n,j},X_r>u_{n,r}).\]
\end{lemma}

\begin{proof}
If $q>0$, we start by observing that
\begin{multline*}
A_{n,s,\ell}^*(x):=\left\{\aa_{n,s}^{s+\ell}\leq x\right\}\cap \left(\WW_{s,\ell}(B_{n,i}(x))\right)^c \subset\\
\bigcup_{i=s+\ell-q}^{s+\ell-1} R^{(1)}_{n,i}(x) \cup \bigcup_{i=s+\ell-2q}^{s+\ell-1} R^{(2)}_{n,i}(x) \cup \ldots \cup \bigcup_{i=s+\ell-\lfloor \ell/q\rfloor q}^{s+\ell-1} R^{(\lfloor\ell/q\rfloor)}_{n,i}(x) \cup \bigcup_{i=s}^{s+\ell-1}\bigcup_{\kappa>\lfloor\ell/q\rfloor} R^{(\kappa)}_{n,i}(x)
\end{multline*}
since $\bigcup_{i=s}^{s+\ell-\kappa q-1} R^{(\kappa)}_{n,i}(x)\subset \left\{\aa_{n,s}^{s+\ell}>x\right\}$ for any $\kappa\leq\lfloor\ell/q\rfloor$. So,

\[\p(A_{n,s,\ell}^*(x))\leq \sum_{\kappa=1}^{\lfloor\ell/q\rfloor}\sum_{i=s+\ell-\kappa q}^{s+\ell-1}\p\left(R^{(\kappa)}_{n,i}(x)\right)+\sum_{i=s}^{s+\ell-1}\sum_{\kappa>\lfloor\ell/q\rfloor}\p\left(R^{(\kappa)}_{n,i}(x)\right).\]

Now, we note that
\[B_{n,s,\ell}^*(x):=\left\{\aa_{u,s}^{s+\ell}>x\right\}\cap \WW_{s,\ell}(B_{n,i}(x))\subset \bigcup_{j=s}^{s+\ell-1}\bigcup_{r=j+1}^{s+\ell-1} Q^{(0)}_{q,n,j}\cap \{X_r>u_{n,r}\}.\]
This is because no entrance in $A_{n,i}(x)$ during the time period $s,\ldots,s+\ell-1$ implies that there must be at least two distinct clusters during the time period $s,\ldots,s+\ell-1$. Since each cluster ends with an escape, \ie the occurrence of $Q^{(0)}_{q,n,j}$, then this must have happened at some moment $j\in\{s,\ldots,s+\ell-1\}$ which was then followed by another exceedance at some subsequent instant $r>j$ where a new cluster is begun. Consequently, we have
\[\p(B_{n,s,\ell}^*(x))\leq \sum_{j=s}^{s+\ell-1}\sum_{r=j+1}^{s+\ell-1}\p\left(Q^{(0)}_{q,n,j}\cap \{X_r>u_{n,r}\}\right)\]

The result follows now at once since
\begin{align*}
\left|\p(\aa_{n,s}^{s+\ell}\leq x)-\p(\WW_{s,\ell}(B_{n,i}(x)))\right|&\leq\p\left(\left\{\aa_{n,s}^{s+\ell}\leq x\right\}\triangle \WW_{s,\ell}(B_{n,i}(x))\right)\\
&=\p(A_{n,s,\ell}^*(x))+\p(B_{n,s,\ell}^*(x))
\end{align*}

If $q=0$, we start by observing that $\{\aa_{n,s}^{s+\ell}\leq x\}\subset\WW_{s,\ell}(B_{n,i}(x))$. Then, we note that
\[\left\{\aa_{n,s}^{s+\ell}>x\right\}\cap \WW_{s,\ell}(B_{n,i}(x))\subset \bigcup_{j=s}^{s+\ell-1}\bigcup_{r=j+1}^{s+\ell-1} \{X_j>u_{n,j}\}\cap\{X_r>u_{n,r}\}.\]
This is because no entrance in $B_{n,i}(x)$ for $i\in\{s,\ldots,s+\ell-1\}$ implies that there must be at least two exceedances during the time period $s,\ldots,s+\ell-1$.

Consequently, we have
\begin{align*}
\left|\p(\aa_{n,s}^{s+\ell}\leq x)-\p\left(\WW_{s,\ell}(B_{n,i}(x))\right)\right|&=\p\left(\left\{\aa_{n,s}^{s+\ell}>x\right\}\cap\WW_{s,\ell}(B_{n,i}(x))\right)\\
&\leq \sum_{j=s}^{s+\ell-1}\sum_{r=j+1}^{s+\ell-1}\p(X_j>u_{n,j},X_r>u_{n,r})
\end{align*}
\end{proof}

As a consequence we obtain an approximation for the Laplace transform of $\aa_{n,s}^{s+\ell}$.
\begin{corollaryP}
\label{cor:exponential}
For any $s,\ell\in\N$, $y\geq 0$ and $n$ sufficiently large we have
\begin{align*}
\Bigg|\E\left(\e^{-y a_n\aa_{n,s}^{s+\ell}}\right)-&\left(1-\sum_{j=s}^{s+\ell-1}\int_0^{\infty}y\e^{-yx}\p(A_{n,j}(x/a_{n,j}))dx\right)\Bigg|\\
&\quad \leq 2\sum_{j=s}^{s+\ell-1}\sum_{r=j+1}^{s+\ell-1}\p\left(Q^{(0)}_{q,n,j}\cap \{X_r>u_{n,r}\}\right)+\int_0^{\infty}y\e^{-yx}\delta_{n,s,\ell}(x/a_n)dx
\end{align*}
\end{corollaryP}

\begin{proof}
Using Lemmas~\ref{Lem:disc-ring}-\ref{lem:entrances-ball-depth}, for every $x>0$ we have when $q>0$
\begin{align*}
\Bigg|\p&\big(\aa_{n,s}^{s+\ell}\leq x\big)-\left(1-\sum_{i=s}^{s+\ell-1}\p(A_{n,i}(x))\right)\Bigg|\leq \left|\p\big(\aa_{n,s}^{s+\ell}\leq x\big)-\p(\WW_{s,\ell}(B_{n,i}(x)))\right|\\
&\quad+\left|\p(\WW_{s,\ell}(B_{n,i}(x))-\p(\WW_{s,\ell}(A_{n,i}(x)))\right|+\left|\p(\WW_{s,\ell}(A_{n,i}(x)))-\left(1-\sum_{i=s}^{s+\ell-1}\p(A_{n,i}(x))\right)\right|\\
&\leq \sum_{j=s}^{s+\ell-1}\sum_{r=j+1}^{s+\ell-1}\p\left(Q^{(0)}_{q,n,j}\cap \{X_r>u_{n,r}\}\right)+\sum_{\kappa=1}^{\lfloor\ell/q\rfloor}\sum_{i=s+\ell-\kappa q}^{s+\ell-1}\p\left(R^{(\kappa)}_{n,i}(x)\right)+\sum_{i=s}^{s+\ell-1}\sum_{\kappa>\lfloor\ell/q\rfloor}\p\left(R^{(\kappa)}_{n,i}(x)\right)\\
&\quad+\sum_{i=1}^q \p(B_{n,s+\ell-i}(x))+\sum_{j=s}^{s+\ell-1}\sum_{r=j+1}^{s+\ell-1}\p\left(Q^{(0)}_{q,n,j}\cap \{X_r>u_{n,r}\}\right)\\
&=2\sum_{j=s}^{s+\ell-1}\sum_{r=j+1}^{s+\ell-1}\p\left(Q^{(0)}_{q,n,j}\cap \{X_r>u_{n,r}\}\right)+\delta_{n,s,\ell}(x)
\end{align*}

When $q=0$, we have
\begin{align*}
\Bigg|\p&\big(\aa_{n,s}^{s+\ell}\leq x\big)-\left(1-\sum_{i=s}^{s+\ell-1}\p(A_{n,i}(x))\right)\Bigg|\leq \left|\p\big(\aa_{n,s}^{s+\ell}\leq x\big)-\p(\WW_{s,\ell}(B_{n,i}(x)))\right|\\
&\quad+\left|\p(\WW_{s,\ell}(B_{n,i}(x))-\p(\WW_{s,\ell}(A_{n,i}(x)))\right|+\left|\p(\WW_{s,\ell}(A_{n,i}(x)))-\left(1-\sum_{i=s}^{s+\ell-1}\p(A_{n,i}(x))\right)\right|\\
&\leq \sum_{j=s}^{s+\ell-1}\sum_{r=j+1}^{s+\ell-1}\p(X_j>u_{n,j},X_r>u_{n,r})+\sum_{j=s}^{s+\ell-1}\sum_{r=j+1}^{s+\ell-1}\p\left(Q^{(0)}_{0,n,j}\cap \{X_r>u_{n,r}\}\right)\\
&=2\sum_{j=s}^{s+\ell-1}\sum_{r=j+1}^{s+\ell-1}\p\left(Q^{(0)}_{0,n,j}\cap \{X_r>u_{n,r}\}\right)+\delta_{n,s,\ell}(x)
\end{align*}

Since $\p(\aa_{n,s}^{s+\ell}<0)=0$, using integration by parts we have
\begin{align*}
\E&\left(\e^{-y a_n\aa_{n,s}^{s+\ell}}\right)=\e^{-y.0}\p(\aa_{n,s}^{s+\ell}=0)+\int_0^{\infty} \e^{-yx} d\p(\aa_{n,s}^{s+\ell}\leq x/a_n)\\
&=\p(\aa_{n,s}^{s+\ell}=0)+\lim_{x\rightarrow\infty}\left[\e^{-yx}\p(\aa_{n,s}^{s+\ell}\leq x/a_n)-\e^{-y.0}\p(\aa_{n,s}^{s+\ell}\leq 0)\right]-\int_0^{\infty}\p(\aa_{n,s}^{s+\ell}\leq x/a_n) d\e^{-yx}\\
&=\p(\aa_{n,s}^{s+\ell}=0)-\p(\aa_{n,s}^{s+\ell}\leq 0)-\int_0^{\infty}(-y\e^{-yx})\p(\aa_{n,s}^{s+\ell}\leq x/a_n)dx\\
&=\int_0^{\infty}y\e^{-yx}\p(\aa_{n,s}^{s+\ell}\leq x/a_n)dx
\end{align*}
Then, using the assumption that $\p(A_{n,j}(x/a_{n,j}))=\p(A_{n,j}(x/a_n))$,
\begin{align*}
&\Bigg|\E\left(\e^{-y a_n\aa_{n,s}^{s+\ell}}\right)-\left(1-\sum_{j=s}^{s+\ell-1}\int_0^{\infty}y\e^{-yx}\p(A_{n,j}(x/a_{n,j}))dx\right)\Bigg|\\
&=\Bigg|\E\left(\e^{-y a_n\aa_{n,s}^{s+\ell}}\right)-\left(1-\sum_{j=s}^{s+\ell-1}\int_0^{\infty}y\e^{-yx}\p(A_{n,j}(x/a_n))dx\right)\Bigg|\\
&=\left|\int_0^{\infty}y\e^{-yx}\p(\aa_{n,s}^{s+\ell}\leq x/a_n)dx-\int_0^{\infty}y\e^{-yx}\left(1-\sum_{j=s}^{s+\ell-1}\int_0^{\infty}y\e^{-yx}\p(A_{n,j}(x/a_n))\right)dx\right|\\
&\leq \int_0^{\infty} y\e^{-yx}\left[
2\sum_{j=s}^{s+\ell-1}\sum_{r=j+1}^{s+\ell-1}\p\left(Q^{(0)}_{q,n,j}\cap \{X_r>u_{n,r}\}\right)+\delta_{n,s,\ell}(x/a_n)\right]dx\\
&=2\sum_{j=s}^{s+\ell-1}\sum_{r=j+1}^{s+\ell-1}\p\left(Q^{(0)}_{q,n,j}\cap \{X_r>u_{n,r}\}\right)+\int_0^{\infty}y\e^{-yx}\delta_{n,s,\ell}(x/a_n)dx
\end{align*}
\end{proof}

Next result gives the main induction tool to build the proof of Theorem~\ref{thm:convergence}.
\begin{lemma}
\label{prop:main-step}
Let $s,\ell,t,\varsigma\in\N$ and consider $x_1\in\R^+_0$, $\underline{x}=(x_2,\ldots,x_\varsigma)\in (\R^+_0)^{\varsigma-1}$, $s+\ell-1+t<a_2<b_2<a_3<\ldots<b_{\varsigma-1}<a_\varsigma<b_\varsigma\in\N_0$. For $n$ sufficiently large we have
\begin{align*}
\big|\p(\aa_{n,s}^{s+\ell}\leq x_1, \aa_{n,a_2}^{b_2}&\leq x_2, \ldots, \aa_{n,a_\varsigma}^{b_\varsigma}\leq x_\varsigma)-\p(\aa_{n,s}^{s+\ell}\leq x_1)\p(\aa_{n,a_2}^{b_2}\leq x_2, \ldots, \aa_{n,a_\varsigma}^{b_\varsigma}\leq x_\varsigma)\big|\\&\hspace{0.4cm} \leq \ell\iota(n,t)+4\sum_{j=s}^{s+\ell-1}\sum_{r=j+1}^{s+\ell-1}\p\left(Q^{(0)}_{q,n,j}\cap\{X_r>u_{n,r}\}\right)+2\delta_{n,s,\ell}(x_1)
\end{align*}
where \begin{equation}
\label{eq:def-iota}
\iota(n,t)=\sup_{s,\ell\in\N}\max_{i=s,\ldots,s+\ell-1}\left\{\left|\p(A_{n,i}(x_1))\p\big(\cap_{j=2}^\varsigma \{\aa_{n,a_j}^{b_j}\leq x_j\}\big)-\p\big(\cap_{j=2}^\varsigma \{\aa_{n,a_j}^{b_j}\leq x_j\}\cap A_{n,i}(x_1)\big)\right|\right\}.
\end{equation}

\end{lemma}
\begin{proof}
Let
\begin{align*}
A_{x_1,\underline{x}}&:=\{\aa_{n,s}^{s+\ell}\leq x_1, \aa_{n,a_2}^{b_2}\leq x_2, \ldots, \aa_{n,a_\varsigma}^{b_\varsigma}\leq x_\varsigma\},\;\;\qquad B_{x_1}:=\{\aa_{n,s}^{s+\ell}\leq x_1\}\\
\tilde A_{x_1,\underline{x}} &:=\WW_{s,\ell}(A_{n,i}(x_1))\cap\{ \aa_{n,a_2}^{b_2}\leq x_2, \ldots, \aa_{n,a_\varsigma}^{b_\varsigma}\leq x_\varsigma\},\quad
\tilde B_{x_1}:=\WW_{s,\ell}(A_{n,i}(x_1)),\\
D^{\underline{x}}&:=\{\aa_{n,a_2}^{b_2}\leq x_2, \ldots, \aa_{n,a_\varsigma}^{b_\varsigma}\leq x_\varsigma\}.
\end{align*}

If $x_1>0$, by Lemmas~\ref{Lem:disc-ring} and \ref{lem:entrances-ball-depth} we have
\begin{align}
\label{eq:approx1}
\left|\p(B_{x_1})-\p(\tilde B_{x_1})\right|\nonumber &\leq \left|\p(\aa_{n,s}^{s+\ell}\leq x_1)-\p(\WW_{s,\ell}(B_{n,i}(x_1)))\right|+\left|\p(\WW_{s,\ell}(B_{n,i}(x_1)))-\p(\WW_{s,\ell}(A_{n,i}(x_1)))\right|\nonumber\\
&\leq \left|\p(\{\aa_{n,s}^{s+\ell}\leq x_1\}\triangle \WW_{s,\ell}(B_{n,i}(x_1)))\right|+\left|\p(\WW_{s,\ell}(A_{n,i}(x_1))\setminus\WW_{s,\ell}(B_{n,i}(x_1)))\right|\nonumber\\
&\leq \sum_{j=s}^{s+\ell-1}\sum_{r=j+1}^{s+\ell-1}\p\left(Q^{(0)}_{q,n,j}\cap \{X_r>u_{n,r}\}\right)+\sum_{\kappa=1}^{\lfloor\ell/q\rfloor}\sum_{i=s+\ell-\kappa q}^{s+\ell-1}\p\left(R^{(\kappa)}_{n,i}(x_1)\right)\nonumber\\
&\quad+\sum_{i=s}^{s+\ell-1}\sum_{\kappa>\lfloor\ell/q\rfloor}\p\left(R^{(\kappa)}_{n,i}(x_1)\right)+\sum_{i=1}^q \p(B_{n,s+\ell-i}(x_1))\nonumber\\
&=\sum_{j=s}^{s+\ell-1}\sum_{r=j+1}^{s+\ell-1}\p\left(Q^{(0)}_{q,n,j}\cap \{X_r>u_{n,r}\}\right)+\delta_{n,s,\ell}(x_1)
\end{align}
and also
{
\fontsize{10}{10}\selectfont
\begin{align}
\label{eq:approx2}
\left|\p(A_{x_1})-\p(\tilde A_{x_1})\right|\nonumber &\leq \left|\p(\{\aa_{n,s}^{s+\ell}\leq x_1\}\cap D^{\underline{x}})-\p(\WW_{s,\ell}(B_{n,i}(x_1))\cap D^{\underline{x}})\right|+\left|\p\left((\WW_{s,\ell}(A_{n,i}(x_1))\setminus\WW_{s,\ell}(B_{n,i}(x_1)))\cap D^{\underline{x}}\right)\right|\nonumber\\
&\leq \left|\p\left((\{\aa_{n,s}^{s+\ell}\leq x_1\}\triangle \WW_{s,\ell}(B_{n,i}(x_1))\cap D^{\underline{x}}\right)\right|+\left|\p\left((\WW_{s,\ell}(A_{n,i}(x_1))\setminus\WW_{s,\ell}(B_{n,i}(x_1)))\cap D^{\underline{x}}\right)\right|\nonumber\\
&\leq \left|\p(\{\aa_{n,s}^{s+\ell}\leq x_1\}\triangle\WW_{s,\ell}(B_{n,i}(x_1))\right|+\left|\p(\WW_{s,\ell}(A_{n,i}(x_1))\setminus\WW_{s,\ell}(B_{n,i}(x_1)))\right|\nonumber\\
&\leq \sum_{j=s}^{s+\ell-1}\sum_{r=j+1}^{s+\ell-1}\p\left(Q^{(0)}_{q,n,j}\cap \{X_r>u_{n,r}\}\right)+\delta_{n,s,\ell}(x_1)
\end{align}
}
If $x_1=0$, we notice that $\{\aa_{n,s}^{s+\ell}\leq x_1\}=\{\aa_{n,s}^{s+\ell}=0\}=\{X_s\leq u_{n,s},\ldots,X_{s+\ell-1}\leq u_{n,s+\ell-1}\}=\WW_{s,\ell}(B_{n,i}(0))$, so estimates \eqref{eq:approx1} and \eqref{eq:approx2} are still valid by Lemma~\ref{Lem:disc-ring}.

Adapting the proof of Lemma~\ref{lem:no-entrances-ring}, it follows that

$\left|\p(\tilde A_{x_1,\underline{x}})-\left(1-\sum_{i=s}^{s+\ell-1}\p(A_{n,i}(x))\right)\p(D^{\underline x})\right|\leq Err,$ where

\[Err=\left|\sum_{i=s}^{s+\ell-1}\p(A_{n,i}(x))\p(D^{\underline{x}})-\sum_{i=s}^{s+\ell-1}\p(A_{n,i}(x)\cap D^{\underline{x}})\right|+\sum_{j=s}^{s+\ell-1}\sum_{r=j+1}^{s+\ell-1}\p\left(Q^{(0)}_{q,n,j}\cap\{X_r>u_{n,r}\}\right)\]

Now, since, by definition of $\iota(n,t)$,
\begin{multline*}
\left|\sum_{i=s}^{s+\ell-1}\p(A_{n,i}(x))\p(D^{\underline{x}})-\sum_{i=s}^{s+\ell-1}\p(A_{n,i}(x)\cap D^{\underline{x}})\right|\\
\leq\sum_{i=s}^{s+\ell-1}\left|\p(A_{n,i}(x))\p(D^{\underline{x}})-\p(A_{n,i}(x)\cap D^{\underline{x}})\right|\leq \ell\iota(n,t),
\end{multline*}

we conclude that
\begin{equation}
\label{eq:approx3}
\left|\p(\tilde A_{x_1,\underline{x}})-\left(1-\sum_{i=s}^{s+\ell-1}\p(A_{n,i}(x))\right)\p(D^{\underline x})\right|\leq \ell\iota(n,t)+\sum_{j=s}^{s+\ell-1}\sum_{r=j+1}^{s+\ell-1}\p\left(Q^{(0)}_{q,n,j}\cap\{X_r>u_{n,r}\}\right)
\end{equation}

Also, by Lemma~\ref{lem:no-entrances-ring} we have
\begin{equation}
\label{eq:approx4}
\left|\p(\tilde B_{x_1})\p(D^{\underline x})-\left(1-\sum_{i=s}^{s+\ell-1}\p(A_{n,i}(x_1))\right)\p(D^{\underline x})\right|\leq \sum_{j=s}^{s+\ell-1}\sum_{r=j+1}^{s+\ell-1}\p\left(Q^{(0)}_{q,n,j}\cap\{X_r>u_{n,r}\}\right)
\end{equation}

Putting together the estimates \eqref{eq:approx1}-\eqref{eq:approx4} we get
\begin{align*}
|\p( & A_{x_1,\underline{x}})-\p(B_{x_1})\p(D^{\underline x})|\leq \left|\p( A_{x_1,\underline{x}})-\p( \tilde A_{x_1,\underline{x}})\right|+\left|\p(\tilde A_{x_1,\underline{x}})-\left(1-\sum_{i=s}^{s+\ell-1}\p(A_{n,i}(x_1))\right)\p(D^{\underline x})\right|\\
&\hspace{2cm}+\left|\p(\tilde B_{x_1})\p(D^{\underline x})-\left(1-\sum_{i=s}^{s+\ell-1}\p(A_{n,i}(x_1))\right)\p(D^{\underline x})\right|+\left|\p(B_{x_1})-\p(\tilde B_{x_1})\right|\p(D^{\underline x})\\
&\hspace{2cm}\leq \ell\iota(n,t)+4\sum_{j=s}^{s+\ell-1}\sum_{r=j+1}^{s+\ell-1}\p\left(Q^{(0)}_{q,n,j}\cap\{X_r>u_{n,r}\}\right)+2\delta_{n,s,\ell}(x_1)
\end{align*}
\end{proof}

Let us consider a function $F:(\R_0^+)^n \to \R$ which is continuous on the right in each variable separately and such that for each $R=(a_1,b_1]\times\ldots\times(a_n,b_n]\subset(\R_0^+)^n$ we have
\[\mu_F(R):=\sum_{c_i\in\{a_i,b_i\}}(-1)^{\#\{i\in\{1,\ldots,n\}:c_i=a_i\}}F(c_1,\ldots,c_n)\geq 0\]
Such $F$ is called an \emph{$n$-dimensional Stieltjes measure function} and such $\mu_F$ has a unique extension to the Borel $\sigma$-algebra in $(\R_0^+)^n$, which is called the \emph{Lebesgue-Stieltjes measure} associated to $F$.

For each $I\subset\{1,\ldots,n\}$, let $F_I(\underline{x}):=F(\delta_1 x_1,\ldots,\delta_n x_n)$, where $\delta_i=
\begin{cases}
1 & \text{if $i\in I$}\\
0 & \text{if $i\notin I$}
\end{cases}$

If $F$ is an $n$-dimensional Stieltjes measure function, it is easy to see that $F_I$ is also an $n$-dimensional Stieltjes measure function, which has an associated Lebesgue-Stieltjes measure $\mu_{F_I}$. We will use the following proposition, proved in \cite[Section~4]{FFM18}:

\begin{propositionP}
\label{prop:Laplace}
Given $n\in\N$, $I\subset\{1,\ldots,n\}$ and two functions $F,G:(\R_0^+)^n \to \R$ such that $F$ is a bounded $n$-dimensional Stieltjes measure function, let
\[\int G(\underline{x})dF_I(\underline{x}):=
\begin{cases}
G(0,\ldots,0)F(0,\ldots,0) & \text{for $I=\emptyset$}\\
\int G(\underline{x})d\mu_{F_I} & \text{for $I\neq\emptyset$}
\end{cases}\]
where $\mu_{F_I}$ is the Lebesgue-Stieltjes measure associated to $F_I$. Then,
{
\fontsize{10}{10}\selectfont
\[\int_0^\infty\!\!\!\!\!\!\ldots\!\!\int_0^\infty \e^{-y_1 x_1-\ldots-y_n x_n}F(\underline{x})dx_1\ldots dx_n=\frac{1}{y_1\ldots y_n}\sum_{I\subset\{1,\ldots,n\}}\int \e^{-\sum_{i\in I} y_i x_i}dF_I(\underline{x})\]
}
\end{propositionP}

\begin{corollaryP}
\label{cor:fgm-main-estimate}
Let $s,\ell,t,\varsigma\in\N$ and consider $y_1,y_2,\ldots,y_\varsigma\in\R_0^+$, $s+\ell-1+t<a_2<b_2<a_3<\ldots<b_{\varsigma-1}<a_\varsigma<b_\varsigma\in\N_0$. For $n$ sufficiently large we have
\[\E\left(\e^{-y_1 a_n\aa_{n,s}^{s+\ell}-y_2 a_n\aa_{n,a_2}^{b_2}-\ldots-y_\varsigma a_n\aa_{n,a_\varsigma}^{b_\varsigma}}\right)=\E\left(\e^{-y_1 a_n\aa_{n,s}^{s+\ell}}\right)\E\left(\e^{-y_2 a_n\aa_{n,a_2}^{b_2}-\ldots-y_\varsigma a_n\aa_{n,a_\varsigma}^{b_\varsigma}}\right)+Err\]
where
\[|Err|\leq \ell\iota(n,t)+4\sum_{j=s}^{s+\ell-1}\sum_{r=j+1}^{s+\ell-1}\p\left(Q^{(0)}_{q,n,j}\cap\{X_r>u_{n,r}\}\right)+2\int_0^{\infty}y_1\e^{-y_1 x}\delta_{n,s,\ell}(x/a_n)dx\]
and $\iota(n,t)$ is given by \eqref{eq:def-iota}.
\end{corollaryP}
\begin{proof}
Using the same notation as in the proof of Lemma~\ref{prop:main-step}, let $F^{(A)}(x_1,\ldots,x_\varsigma)=\p(A_{x_1,\underline{x}})$, $F^{(B)}(x_1)=\p(B_{x_1})$ and $F^{(D)}(x_2,\ldots,x_\varsigma)=\p(D^{\underline{x}})$. Then, $F^{(A)}$, $F^{(B)}$ and $F^{(D)}$ are both bounded Stieltjes measure functions, with
\begin{align*}
\mu_{F^{(A)}}(U_1)&=\p\left((a_n\aa_{n,s}^{s+\ell}, a_n\aa_{n,a_2}^{b_2},\ldots, a_n\aa_{n,a_\varsigma}^{b_\varsigma})\in U_1\right)\\
\mu_{F^{(B)}}(U_2)&=\p(a_n\aa_{n,s}^{s+\ell}\in U_2)\qquad
\mu_{F^{(D)}}(U_3)=\p\left(( a_n\aa_{n,a_2}^{b_2},\ldots, a_n\aa_{n,a_\varsigma}^{b_\varsigma})\in U_3\right)
\end{align*}
where $U_1$, $U_2$ and $U_3$ are Borel sets in $(\R_0^+)^\varsigma$, $\R_0^+$ and $(\R_0^+)^{\varsigma-1}$, respectively.

Therefore, we can apply the previous proposition and we obtain
{
\fontsize{10}{10}\selectfont
\begin{align*}
&\E\left(\e^{-y_1 a_n\aa_{n,s}^{s+\ell}-y_2 a_n\aa_{n,a_2}^{b_2}-\ldots-y_\varsigma a_n\aa_{n,a_\varsigma}^{b_\varsigma}}\right)-\E\left(\e^{-y_1 a_n\aa_{n,s}^{s+\ell}}\right)\E\left(\e^{- y_2 a_n\aa_{n,a_2}^{b_2}-\ldots-y_\varsigma a_n\aa_{n,a_\varsigma}^{b_\varsigma}}\right)\\
&=\sum_{I\subset\{1,\ldots,\varsigma\}}\int\e^{-\sum_{i\in I}y_i a_n x_i}d(F^{(A)})_I(x_1,\ldots,x_\varsigma)\\
&-\sum_{I\subset\{1\}}\int\e^{-\sum_{i\in I}y_i a_n x_i}d(F^{(B)})_I(x_1)\sum_{I\subset\{2,\ldots,\varsigma\}}\int\e^{-\sum_{i\in I}y_i a_n x_i}d(F^{(D)})_I(x_2,\ldots,x_\varsigma)\\
&=y_1\ldots y_{\varsigma}a_n^{\varsigma}\int_0^\infty\!\!\!\!\!\!\ldots\!\!\int_0^\infty\e^{-y_1 a_n x_1-\ldots-y_{\varsigma} a_n x_{\varsigma}}F^{(A)}(x_1,\ldots,x_\varsigma)dx_1\ldots dx_{\varsigma}\\
&-\left(y_1 a_n\int_0^\infty\e^{-y_1 a_n x_1}F^{(B)}(x_1)dx_1\right)\left(y_2\ldots y_{\varsigma}a_n^{\varsigma-1}\int_0^\infty\!\!\!\!\!\!\ldots\!\!\int_0^\infty\e^{-y_2 a_n x_2-\ldots-y_{\varsigma} a_n x_{\varsigma}}F^{(D)}(x_2,\ldots,x_\varsigma)dx_2\ldots dx_{\varsigma}\right)\\
&=y_1\ldots y_{\varsigma}a_n^{\varsigma}\int_0^\infty\!\!\!\!\!\!\ldots\!\!\int_0^\infty\e^{-y_1 a_n x_1-\ldots-y_{\varsigma} a_n x_{\varsigma}}(F^{(A)}-F^{(B)}F^{(D)})(x_1,\ldots,x_\varsigma)dx_1\ldots dx_{\varsigma}
\end{align*}
}
Hence, using Lemma~\ref{prop:main-step},
\begin{multline*}
\left|\E\left(\e^{-y_1 a_n\aa_{n,0}^s-y_2 a_n\aa_{n,a_2}^{b_2}-\ldots-y_\varsigma a_n\aa_{n,a_\varsigma}^{b_\varsigma})}\right)-\E\left(\e^{-y_1 a_n\aa_{n,0}^s}\right)\E\left(\e^{- y_2 a_n\aa_{n,a_2}^{b_2}-\ldots-y_\varsigma a_n\aa_{n,a_\varsigma}^{b_\varsigma}}\right)\right|\\
\leq y_1\ldots y_{\varsigma}a_n^{\varsigma}\int_0^\infty\!\!\!\!\!\!\ldots\!\!\int_0^\infty\e^{-y_1 a x_1-\ldots-y_{\varsigma} a x_{\varsigma}}\left|\p(A_{x_1,\underline{x}})-\p(B_{x_1})\p(D^{\underline{x}})\right|dx_1\ldots dx_{\varsigma}\\
\leq \ell\iota(n,t)+4\sum_{j=s}^{s+\ell-1}\sum_{r=j+1}^{s+\ell-1}\p\left(Q^{(0)}_{q,n,j}\cap\{X_r>u_{n,r}\}\right)+2y_1 a_n\int_0^\infty\e^{-y_1 a_n x_1}\delta_{n,s,\ell}(x_1)dx_1\\
=\ell\iota(n,t)+4\sum_{j=s}^{s+\ell-1}\sum_{r=j+1}^{s+\ell-1}\p\left(Q^{(0)}_{q,n,j}\cap\{X_r>u_{n,r}\}\right)+2\int_0^\infty y_1\e^{-y_1 x}\delta_{n,s,\ell}(x/a_n)dx
\end{multline*}
\end{proof}

\begin{propositionP}
\label{prop:Laplace-estimates}
Let $X_0,X_1,\ldots$ be given by \eqref{eq:def-stat-stoch-proc-DS}, let $J\in\RR$ be such that $J=\bigcup_{\ell=1}^\varsigma I_\ell$ where $I_j=[a_j,b_j)\in\S$, $j=1,\ldots,\varsigma$ and $a_1<b_1<a_2<\cdots<b_{\varsigma-1}<a_\varsigma<b_\varsigma$, let $u_{n,i}$ be real-valued boundaries satisfying \eqref{FmaxH} and \eqref{F_hn}, let $H:=\lceil\sup\{x:x\in J\}\rceil=\lceil b_\varsigma\rceil$ and let $(a_n)_{n\in\N}$ be a normalising sequence, $a_{n,j}$ normalising factors and $\pi$ a probability distribution as in \eqref{thm:convergence}. Assume that $\D_q(u_{n,i})^*$, $\D'_q(u_{n,i})^*$ and $U\!LC_q(u_{n,i})^*$ hold, for some $q\in\N_0$. Consider the partition of $[0,Hn)$ into blocks of length $\ell_{H,n,j}$, $J_1=[\L_{H,n,0},\L_{H,n,1})$, $J_2=[\L_{H,n,1},\L_{H,n,2})$, ..., $J_{k_n}=[\L_{H,n,k_n-1},\L_{H,n,k_n})$, $J_{k_n+1}=[\L_{H,n,k_n},Hn)$. Let $n$ be sufficiently large so that $L_{H,n}:=\max\{\ell_{H,n,j},j=1,\ldots,k_n\}<\frac{n}{2}\inf_{j\in \{1,\ldots,\varsigma\}}\{b_j-a_j\}$ and, finally, let $\mathscr S_\ell$ be the number of blocks $J_i,i>1$ contained in $n I_\ell$, that is,
\[\mathscr S_\ell:=\#\{i\in \{2,\ldots,k_n\}:J_i\subset n I_\ell\}.\]
Note that, by definition of $L_{H,n}$, we must have $\mathscr S_\ell>1$ for every $\ell\in \{1,\ldots,\varsigma\}$.

Then, for all $y_1,y_2,\ldots,y_\varsigma\in\R_0^+$, we have
\[\E\left(\e^{-\sum_{\ell=1}^\varsigma y_\ell a_n\aa_n(nI_\ell)}\right)-\prod_{\ell=1}^\varsigma\prod_{i=i_{\ell}}^{i_\ell+\mathscr S_\ell-1} \E\Big(\e^{-y_{\ell}a_n\aa_n(J_i)}\Big) \xrightarrow[n\to\infty]{}0\]
\end{propositionP}

\begin{proof}
Without loss of generality, we can assume that $y_1,y_2,\ldots,y_\varsigma\in\R^+$, because if we had $y_j=0$ for some $j=1,\ldots,\varsigma$ then we could consider $J=\bigcup_{\ell=1}^{j-1}I_\ell\cup\bigcup_{\ell=j+1}^\varsigma I_\ell$ instead. Also, we can assume that $a_1>0$.
Let $\hat y:=\inf\{y_j:j=1,\ldots,\varsigma\}>0$ and $\hat Y:=\sup\{y_j:j=1,\ldots,\varsigma\}$. We cut each $J_i$ into two blocks:
\[J_i^*:=[\L_{H,n,i-1},\L_{H,n,i}-t_{H,n,i})\; \mbox{ and } J_i':=J_i\setminus J_i^*\]
Note that $|J_i^*|=\ell_{H,n,i}-t_{H,n,i}$ and $|J_i'|=t_{H,n,i}$.

For each $\ell\in \{1,\ldots,\varsigma\}$, we define $i_\ell:=\min\{i\in \{2,\ldots,k_n\}:J_i\subset n I_{\ell}\}$.
Hence, it follows that $J_{i_\ell},J_{i_\ell+1},\ldots,J_{i_\ell+\mathscr S_\ell-1}\subset n I_\ell$ and
\begin{equation}
\label{eq:blocks}
\L_{H,n,i_\ell+\mathscr S_\ell-1}-\L_{H,n,i_\ell-1}=\sum_{j=i_\ell}^{i_\ell+\mathscr S_\ell-1}\ell_{H,n,j}\sim n|I_\ell|
\end{equation}

First of all, recall that for every $0\leq x_i, z_i\leq 1$, we have
\begin{equation}
\label{eq:inequality}
\left|\prod x_i-\prod z_i\right|\leq \sum |x_i-z_i|.
\end{equation}

We start by making the following approximation, in which we use \eqref{eq:inequality},
{\fontsize{10}{10}\selectfont
\begin{align*}
\Bigg|\E\left(\e^{-\sum_{\ell=1}^\varsigma y_\ell a_n\aa_n(nI_\ell)}\right)&-\E\left(\e^{-\sum_{\ell=1}^\varsigma y_\ell\sum_{j=i_\ell}^{i_\ell+\mathscr S_\ell-1}a_n\aa_n(J_j)}\right)\Bigg|\leq\E\left(1-\e^{-\sum_{\ell=1}^\varsigma y_\ell a_n\aa_n\left(nI_\ell\setminus\cup_{j=i_\ell}^{i_\ell+\mathscr S_\ell-1}J_j\right)}\right)\\
&\leq \E\left(1-\e^{-\sum_{\ell=1}^\varsigma y_\ell a_n\aa_n(J_{i_\ell-1}\cup J_{i_\ell+\mathscr S_\ell})}\right)\\
&\leq \varsigma K\E\left(1-\e^{-a_n\aa_n(J_{i_\ell-1})}\right)+\varsigma K\E\left(1-\e^{-a_n\aa_n(J_{i_\ell+\mathscr S_\ell})}\right),
\end{align*}
}

where $\max\{y_1,\ldots,y_\varsigma\}\leq K\in\N$.
In order to show that we are allowed to use the above approximation we just need to check that $\E\left(1-\e^{-a_n\aa_n(J_i)}\right)\to0$ as $n\to\infty$ for every $i=1,\ldots,k_n+1$.
By Corollary~\ref{cor:exponential} we have for $i=1,\ldots,k_n$
\begin{equation}
\label{eq:error1}
\E\left(\e^{-a_n\aa_n(J_i)}\right)=\E\left(\e^{-a_n\aa_{n,\L_{H,n,i-1}}^{\L_{H,n,i}}}\right)=1-\sum_{j=\L_{H,n,i-1}}^{\L_{H,n,i}-1}\int_0^{\infty}\e^{-x}\p(A_{n,j}(x/a_{n,j}))dx+Err,
\end{equation}
where
\[\left|Err\right|\leq 2\sum_{j=\L_{H,n,i-1}}^{\L_{H,n,i}-1}\sum_{r=j+1}^{\L_{H,n,i}-1}\p\left(Q^{(0)}_{q,n,j}\cap \{X_r>u_{n,r}\}\right)+\int_0^{\infty}\e^{-x}\delta_{n,\L_{H,n,i-1},\ell_{H,n,i}}(x/a_n)dx \to 0\]
as $n\to\infty$ by $\D'_q(u_{n,i})^*$ and $U\!LC_q(u_{n,i})$.
Since
\[\sum_{j=\L_{H,n,i-1}}^{\L_{H,n,i}-1}\int_0^{\infty}\e^{-x}\p(A_{n,j}(x/a_{n,j}))dx \leq \sum_{j=\L_{H,n,i-1}}^{\L_{H,n,i}-1}\int_0^{\infty}\e^{-x}\p(X_j>u_{n,j})dx=\sum_{j=\L_{H,n,i-1}}^{\L_{H,n,i}-1}\bar F(u_{n,j})\leq\frac{F_{H,n}^*}{k_n}\]
we get $\E\left(\e^{-a_n\aa_n(J_i)}\right)\xrightarrow[n\to\infty]{}1$ by \eqref{F_hn}.

If $i=k_n+1$ then
\begin{equation}
\E\left(\e^{-a_n\aa_n(J_i)}\right)=\E\left(\e^{-a_n\aa_{n,\L_{H,n,k_n}}^{Hn}}\right)=1-\sum_{j=\L_{H,n,k_n}}^{Hn-1}\int_0^{\infty}\e^{-x}\p(A_{n,j}(x/a_n))dx+Err,
\end{equation}
where
\[\left|Err\right|\leq 2\sum_{j=\L_{H,n,k_n}}^{Hn-1}\sum_{r=j+1}^{Hn-1}\p\left(Q^{(0)}_{q,n,j}\cap \{X_r>u_{n,r}\}\right)+\int_0^{\infty}\e^{-x}\delta_{n,\L_{H,n,k_n},Hn-\L_{H,n,k_n}}(x/a_n)dx \to 0\]
as $n\to\infty$ by $\D'_q(u_{n,i})^*$ and $U\!LC_q(u_{n,i})$.
Since, by \eqref{eq:block-estimate-last},
\[\sum_{j=\L_{H,n,k_n}}^{Hn-1}\int_0^{\infty}\e^{-x}\p(A_{n,j}(x/a_{n,j}))dx \leq \sum_{j=\L_{H,n,k_n}}^{Hn-1}\int_0^{\infty}\e^{-x}\p(X_j>u_{n,j})dx=\sum_{j=\L_{H,n,k_n}}^{Hn-1}\bar F(u_{n,j})\leq k_n\bar F_{n,\max}(H)\]
we get $\E\left(\e^{-a_n\aa_n(J_{k_n+1})}\right)\xrightarrow[n\to\infty]{}1$ by \eqref{eq:kn-sequence}.

Now, we proceed with another approximation which consists of replacing $J_i$ by $J_i^*$. Using \eqref{eq:inequality} we have
\begin{align*}
\Bigg|\E\left(\e^{-\sum_{\ell=1}^\varsigma y_\ell\sum_{i=i_\ell}^{i_\ell+\mathscr S_\ell-1}a_n\aa_n(J_i)}\right)&-\E\left(\e^{-\sum_{\ell=1}^\varsigma y_\ell\sum_{i=i_\ell}^{i_\ell+\mathscr S_\ell-1}a_n\aa_n(J_i^*)}\right)\Bigg|\leq \E\left(1-\e^{-\sum_{\ell=1}^\varsigma y_\ell\sum_{i=i_\ell}^{i_\ell+\mathscr S_\ell-1}a_n\aa_n(J'_i)}\right)\\
&\leq K\sum_{\ell=1}^\varsigma \E\left(1-\e^{-\sum_{i=i_\ell}^{i_\ell+\mathscr S_\ell-1}a_n\aa_n(J'_i)}\right)\leq K\sum_{\ell=1}^\varsigma \sum_{i=i_\ell}^{i_\ell+\mathscr S_\ell-1}\E\left(1-\e^{-a_n\aa_n(J'_i)}\right)\\
&\leq K\sum_{i=1}^{k_n}\E\left(1-\e^{-a_n\aa_n(J'_i)}\right)
\end{align*}

Now, we must show that $\sum_{i=1}^{k_n}\E\left(1-\e^{-a_n\aa_n(J'_i)}\right)\to 0,$ as $n\to\infty$, in order for the approximation to make sense. By Corollary~\ref{cor:exponential} we have

\begin{equation}
\label{eq:error2}
\E\left(\e^{-a_n\aa_n(J'_i)}\right)=\E\left(\e^{-a_n\aa_{n,\L_{H,n,i}-t_{H,n,i}}^{\L_{H,n,i}}}\right)=1-\sum_{j=\L_{H,n,i}- t_{H,n,i}}^{\L_{H,n,i}-1}\int_0^{\infty}\e^{-x}\p(A_{n,j}(x/a_{n,j}))dx+Err,
\end{equation}
where
\begin{align*}
\sum_{i=1}^{k_n}\left|Err\right|\leq & 2\sum_{i=1}^{k_n}\sum_{j=\L_{H,n,i}-t_{H,n,i}}^{\L_{H,n,i}-1}\sum_{r=j+1}^{\L_{H,n,i}-1}\p\left(Q^{(0)}_{q,n,j}\cap \{X_r>u_{n,r}\}\right)\\
&+\sum_{i=1}^{k_n}\int_0^{\infty}\e^{-x}\delta_{\L_{H,n,i}- t_{H,n,i},t_{H,n,i},n}(x/a_n)dx \to 0
\end{align*}

as $n\to\infty$ by $\D'_q(u_{n,i})^*$ and $U\!LC_q(u_{n,i})$. We get, by \eqref{eq:kn-sequence} as well,
\begin{align*}
\sum_{i=1}^{k_n}&\E\left(1-\e^{-a_n\aa_n(J'_i)}\right)\sim \sum_{i=1}^{k_n}\sum_{j=\L_{H,n,i}-t_{H,n,i}}^{\L_{H,n,i}-1}\int_0^{\infty}\e^{-x}\p(A_{n,j}(x/a_{n,j}))dx\\
&\leq\sum_{i=1}^{k_n}\sum_{j=\L_{H,n,i}-t_{H,n,i}}^{\L_{H,n,i}-1}\bar F(u_{n,j})\leq\sum_{i=1}^{k_n}\varepsilon(H,n)\frac{F_{H,n}^*}{k_n}=k_n(t_n^*+1)\bar F_{n,\max}(H)\xrightarrow[n\to\infty]{}0
\end{align*}

Let us fix now some $\hat\ell\in\{1,\ldots,\varsigma\}$ and $i\in\{i_{\hat\ell},\ldots,i_{\hat\ell}+\mathscr S_{\hat\ell}-1\}$. Let $M_i=y_{\hat\ell}\sum_{j=i}^{i_{\hat\ell}+\mathscr S_{\hat\ell}-1}a_n\aa_n(J_j^*)$ and $L_{\hat\ell}=\sum_{\ell=\hat\ell+1}^\varsigma y_\ell\sum_{j=i_{\ell}}^{i_\ell+\mathscr S_\ell-1}a_n\aa_n(J_j^*).$
Using Corollary~\ref{cor:fgm-main-estimate} along with the facts that $\iota(n,t)\le\gamma_i(n,t)$ and $\gamma_i(n,t)$ is decreasing in $t$, we obtain
\begin{equation*}
\left|\E\left(\e^{-y_{\hat\ell}a_n\aa_n(J_{i_{\hat\ell}}^*)-M_{i_{\hat\ell}+1}-L_{\hat\ell}}\right)-\E\left(\e^{-y_{\hat\ell}a_n\aa_n(J_{i_{\hat\ell}}^*)}\right)\E\left(\e^{-M_{i_{\hat\ell}+1}-L_{\hat\ell}}\right) \right|\leq \Upsilon_{n,i_{\hat\ell}},
\end{equation*}
where
\begin{align*}
\Upsilon_{n,i}&=t_{H,n,i}\gamma_i(n,t_{H,n,i})+4\sum_{j=\L_{H,n,i-1}}^{\L_{H,n,i-1}+t_{H,n,i}-1}\sum_{r=j+1}^{\L_{H,n,i-1}+t_{H,n,i}-1}\p\left(Q^{(0)}_{q,n,j}\cap\{X_r>u_{n,r}\}\right)\\
&\qquad+2\int_0^{\infty}y_{\hat\ell}\e^{-y_{\hat\ell}x}\delta_{n,\L_{H,n,i-1},\ell_{H,n,i}-t_{H,n,i}}(x/a_n)dx\\
&\leq \ell_{H,n,i}\gamma_i(n,t_n^*)+4\sum_{j=\L_{H,n,i-1}}^{\L_{H,n,i}-1}\sum_{r=j+1}^{\L_{H,n,i}-1}\p\left(Q^{(0)}_{q,n,j}\cap\{X_r>u_{n,r}\}\right)\\
&\qquad+2\hat Y\int_0^{\infty}\e^{-\hat y x}\delta_{n,\L_{H,n,i-1},\ell_{H,n,i}-t_{H,n,i}}(x/a_n)dx
\end{align*}

Since $\E\left(\e^{-y_{\hat\ell}a_n\aa_n(J_i^*)}\right)\leq 1$ for any $i\in\{1,\ldots,k_n\}$, it follows by the same argument that
\begin{align*}
\Big|\E\Big(\e^{-M_{i_{\hat\ell}}-L_{\hat\ell}}\Big)-&\E\Big(\e^{-y_{\hat\ell}a_n\aa_n(J_{i_{\hat\ell}}^*)}\Big)\E\Big(\e^{-y_{\hat\ell}a_n\aa_n(J_{i_{\hat\ell}+1}^*)}\Big)\E\Big(\e^{-M_{i_{\hat\ell}+2}-L_{\hat\ell}}\Big)\Big|\\
&\leq\Big|\E\Big(\e^{-M_{i_{\hat\ell}}-L_{\hat\ell}}\Big)-\E\Big(\e^{-y_{\hat\ell}a_n\aa_n(J_{i_{\hat\ell}}^*)}\Big)\E\Big(\e^{-M_{i_{\hat\ell}+1}-L_{\hat\ell}}\Big) \Big|\\
&+\E\Big(\e^{-y_{\hat\ell}a_n\aa_n(J_{i_{\hat\ell}}^*)}\Big)\Big|\E\Big(\e^{-M_{i_{\hat\ell}+1}-L_{\hat\ell}}\Big)-\E\Big(\e^{-y_{\hat\ell}a_n\aa_n(J_{i_{\hat\ell}+1}^*)}\Big)\E\Big(\e^{-M_{i_{\hat\ell}+2}-L_{\hat\ell}}\Big)\Big|\\
&\leq \Upsilon_{n,i_{\hat\ell}}+\Upsilon_{n,i_{\hat\ell}+1}
\end{align*}

Hence, proceeding inductively with respect to $i\in\{i_{\hat\ell},\ldots,i_{\hat\ell}+\mathscr S_{\hat\ell}-1\}$, we obtain
\[\Big|\E\Big(\e^{-M_{i_{\hat\ell}}-L_{\hat\ell}}\Big)-\prod_{j=i_{\hat\ell}}^{i_{\hat\ell}+\mathscr S_{\hat\ell}-1}\E\Big(\e^{-y_{\hat\ell}a_n\aa_n(J_j^*)}\Big)\E\Big(\e^{-L_{\hat\ell}}\Big)\Big|\leq \sum_{i=i_{\hat\ell}}^{i_{\hat\ell}+\mathscr S_{\hat\ell}-1}\Upsilon_{n,i}\]

In the same way, if we proceed inductively with respect to $\hat\ell\in\{1,\ldots,\varsigma\}$, we get
\begin{align*}
&\left|\E\left(\e^{-\sum_{\ell=1}^\varsigma y_\ell\sum_{j=i_\ell}^{i_\ell+\mathscr S_\ell}a_n\aa_n(J_j^*)}\right)-\prod_{\ell=1}^\varsigma\prod_{i=i_{\ell}}^{i_\ell+\mathscr S_\ell-1}\E\Big(\e^{-y_{\ell}a_n\aa_n(J_i^*)}\Big)\right|\leq \sum_{\ell=1}^\varsigma\sum_{i=i_{\ell}}^{i_{\ell}+\mathscr S_{\ell}-1}\Upsilon_{n,i}\\
&\leq\sum_{i=1}^{k_n}\Upsilon_{n,i}\leq Hn\gamma_i(n,t_n^*)+4\sum_{i=1}^{k_n}\sum_{j=\L_{H,n,i-1}}^{\L_{H,n,i}-1}\sum_{r=j+1}^{\L_{H,n,i}-1}\p\left(Q^{(0)}_{q,n,j}\cap\{X_r>u_{n,r}\}\right)\\
&\qquad\qquad\qquad+2\hat Y\sum_{i=1}^{k_n}\int_0^{\infty}\e^{-\hat y x}\delta_{n,\L_{H,n,i-1},\ell_{H,n,i}-t_{H,n,i}}(x/a_n)dx\to 0
\end{align*}
as $n\to\infty$, by $\D_q(u_{n,i})^*$, $\D'_q(u_{n,i})^*$ and $U\!LC_q(u_{n,i})$.

Using \eqref{eq:inequality} again, we have the final approximation
\begin{align*}
\left|\prod_{\ell=1}^\varsigma\prod_{i=i_{\ell}}^{i_\ell+\mathscr S_\ell-1} \E\Big(\e^{-y_{\ell}a_n\aa_n(J_i)}\Big)-\prod_{\ell=1}^\varsigma\prod_{i=i_{\ell}}^{i_\ell+\mathscr S_\ell-1} \E\Big(\e^{-y_{\ell}a_n\aa_n(J_i^*)}\Big)\right| &\leq K\sum_{\ell=1}^\varsigma\sum_{i=i_{\ell}}^{i_\ell+\mathscr S_\ell-1} \E\left(1-\e^{-a_n\aa_n(J'_i)}\right)\\
&\leq K \sum_{i=1}^{k_n}\E\left(1-\e^{-a_n\aa_n(J'_i)}\right).
\end{align*}
We have already proved that $\sum_{i=1}^{k_n}\E\left(1-\e^{-a_n\aa_n(J'_i)}\right)\to 0$ as $n\to\infty$, so we only need to gather all the approximations to finally obtain the stated result.

\end{proof}

\begin{proof}[Proof of Theorem~\ref{thm:convergence}]
In order to prove convergence of $a_n A_n$ to a process $A$, it is sufficient to show that for any $\varsigma$ disjoint intervals $I_1, I_2,\ldots, I_\varsigma\in\S$, the joint distribution of $a_n A_n$ over these intervals converges to the joint distribution of $A$ over the same intervals, \ie
\[
(a_n A_n(I_1),a_n A_n(I_2),\ldots,a_n A_n(I_\varsigma))\xrightarrow[n\to\infty]{}(A(I_1), A(I_2), \ldots, A(I_\varsigma)),
\]
which will be the case if the corresponding joint Laplace transforms converge. Hence, we only need to show that
\[
\psi_{a_n A_n}(y_1, y_2,\ldots,y_\varsigma)\to \psi_A(y_1, y_2,\ldots,y_\varsigma)=\E\left(\e^{-\sum_{\ell=1}^\varsigma y_\ell A(I_\ell)}\right), \quad \text{as $n\to\infty$,}
\]
for every $\varsigma$ non-negative values $y_1, y_2,\ldots,y_\varsigma$, each choice of $\varsigma$ disjoint intervals $I_1, I_2,\ldots, I_\varsigma\in\S$ and each $\varsigma\in\N$. Note that $\psi_{a_n A_n}(y_1, y_2,\ldots,y_\varsigma)=\E\left(\e^{-\sum_{\ell=1}^\varsigma y_\ell a_n A_n(I_\ell)}\right)=\E\left(\e^{-\sum_{\ell=1}^\varsigma y_\ell a_n\aa_n(v_nI_\ell)}\right)$ and
{
\fontsize{10}{10}\selectfont
\begin{align*}
\left|\E\left(\e^{-\sum_{\ell=1}^\varsigma y_\ell a_n \aa_n(v_n I_\ell)}\right)-\E\left(\e^{-\sum_{\ell=1}^\varsigma y_\ell A(I_\ell)}\right)\right|
&\leq\left|\E\left(\e^{-\sum_{\ell=1}^\varsigma y_\ell a_n\aa_n(v_nI_\ell)}\right)-\E\left(\e^{-\sum_{\ell=1}^\varsigma y_\ell a_n\aa_n\left(\frac{n}{\tau}I_\ell\right)}\right)\right|\\
&+\left|\E\left(\e^{-\sum_{\ell=1}^\varsigma y_\ell a_n\aa_n\left(\frac{n}{\tau}I_\ell\right)}\right)-\prod_{\ell=1}^\varsigma\prod_{i=i_{\ell}}^{i_\ell+\mathscr S_\ell-1} \E\Big(\e^{-y_{\ell}a_n\aa_n(J_i)}\Big)\right|\\
&+\left|\prod_{\ell=1}^\varsigma\prod_{i=i_{\ell}}^{i_\ell+\mathscr S_\ell-1} \E\Big(\e^{-y_{\ell}a_n\aa_n(J_i)}\Big)-\E\left(\e^{-\sum_{\ell=1}^\varsigma y_\ell A(I_\ell)}\right)\right|
\end{align*}
}

where $J_1,J_2,\ldots,J_{k_n+1}$ are the elements of the partition of $[0,Hn)$ given by Proposition~\ref{prop:Laplace-estimates}, with $J=\bigcup_{\ell=1}^\varsigma \frac{1}{\tau}I_\ell$. Since $v_n\sim\frac{n}{\tau}$, the first term on the right goes to $0$ as $n\to\infty$. By Proposition~\ref{prop:Laplace-estimates}, the second term on the right also goes to $0$ as $n\to\infty$. Finally, by Corollary~\ref{cor:exponential}, we have
\begin{align*}
\E\left(\e^{-y_\ell a_n\aa_n(J_i)}\right)=1-\sum_{j=\L_{H,n,i-1}}^{\L_{H,n,i}-1}\int_0^{\infty}y_\ell\e^{-y_\ell x}\p(A_{n,j}(x/a_{n,j}))dx+Err
\end{align*}
where
\[\left|Err\right|\leq 2\sum_{j=\L_{H,n,i-1}}^{\L_{H,n,i}-1}\sum_{r=j+1}^{\L_{H,n,i}-1}\p\left(Q^{(0)}_{q,n,j}\cap \{X_r>u_{n,r}\}\right)+\int_0^{\infty}y_\ell\e^{-y_\ell x}\delta_{n,\L_{H,n,i-1},\ell_{H,n,i}}(x/a_n)dx\]

Using \eqref{eq:inequality}, we have that
\begin{align*}
&\left|\prod_{\ell=1}^\varsigma\prod_{i=i_\ell}^{i_\ell+\mathscr S_\ell-1}\E\Big(\e^{-y_\ell a_n\aa_n(J_i)}\Big)-\prod_{\ell=1}^\varsigma\prod_{i=i_{\ell}}^{i_\ell+\mathscr S_\ell-1}\left(1-\sum_{j=\L_{H,n,i-1}}^{\L_{H,n,i}-1}\int_0^{\infty}y_\ell\e^{-y_\ell x}\p(A_{n,j}(x/a_{n,j}))dx\right)\right|\\
&\leq\sum_{\ell=1}^\varsigma\sum_{i=i_\ell}^{i_\ell+\mathscr S_\ell-1}\left|\E\Big(\e^{-y_\ell a_n\aa_n(J_i)}\Big)-\left(1-\sum_{j=\L_{H,n,i-1}}^{\L_{H,n,i}-1}\int_0^{\infty}y_\ell\e^{-y_\ell x}\p(A_{n,j}(x/a_{n,j}))dx\right)\right|\\
&\leq\sum_{\ell=1}^\varsigma\sum_{i=i_\ell}^{i_\ell+\mathscr S_\ell-1}|Err|\leq\sum_{i=1}^{k_n}|Err|\\
&\leq 2\sum_{i=1}^{k_n}\sum_{j=\L_{H,n,i-1}}^{\L_{H,n,i}-1}\sum_{r=j+1}^{\L_{H,n,i}-1}\p\left(Q^{(0)}_{q,n,j}\cap \{X_r>u_{n,r}\}\right)+\sum_{i=1}^{k_n}\int_0^{\infty}y_\ell\e^{-y_\ell x}\delta_{n,\L_{H,n,i-1},\ell_{H,n,i}}(x/a_n)dx\\
&\to 0
\end{align*}

as $n\to\infty$ by $\D'_q(u_{n,i})^*$ and $U\!LC_q(u_{n,i})$, so it follows that
\begin{align*}
\prod_{\ell=1}^\varsigma\prod_{i=i_\ell}^{i_\ell+\mathscr S_\ell-1}&\E\Big(\e^{-y_\ell a_n\aa_n(J_i)}\Big)\sim\prod_{\ell=1}^\varsigma\prod_{i=i_{\ell}}^{i_\ell+\mathscr S_\ell-1}\left(1-\sum_{j=\L_{H,n,i-1}}^{\L_{H,n,i}-1}\int_0^{\infty}y_\ell\e^{-y_\ell x}\p(A_{n,j}(x/a_{n,j}))dx\right)\\
&\sim\prod_{\ell=1}^\varsigma\prod_{i=i_{\ell}}^{i_\ell+\mathscr S_\ell-1}\left(1-\sum_{j=\L_{H,n,i-1}}^{\L_{H,n,i}-1}\p\left(Q^{(0)}_{q,n,j}\right)\int_0^{\infty}y_\ell\e^{-y_\ell x}(1-\pi(x))dx\right)\\
&\sim\prod_{\ell=1}^\varsigma\prod_{i=i_\ell}^{i_\ell+\mathscr S_\ell-1}\left(1-\theta\sum_{j=\L_{H,n,i-1}}^{\L_{H,n,i}-1}\bar F(u_{n,j})\left(1-\pi(0)-\int_0^{\infty}\e^{-y_\ell x}d\pi(x)\right)\right)\\
&=\prod_{\ell=1}^\varsigma\prod_{i=i_\ell}^{i_\ell+\mathscr S_\ell-1}\left(1-\theta(1-\phi(y_\ell))\sum_{j=\L_{H,n,i-1}}^{\L_{H,n,i}-1}\bar F(u_{n,j})\right)\\
&\sim\e^{-\sum_{\ell=1}^\varsigma\sum_{i=i_\ell}^{i_\ell+\mathscr S_\ell-1}\theta(1-\phi(y_\ell))\sum_{j=\L_{H,n,i-1}}^{\L_{H,n,i}-1}\bar F(u_{n,j})}
\end{align*}

where $\phi$ is the Laplace transform of $\pi$, and since we have, by \eqref{F_hn},
\[\left|\sum_{j=\L_{H,n,i-1}}^{\L_{H,n,i}-1}\bar F(u_{n,j})-\frac{\ell_{H,n,i}}{n}\tau\right|\leq\left|\sum_{j=0}^{\L_{H,n,i}-1}\bar F(u_{n,j})-\frac{\L_{H,n,i}}{n}\tau\right|+\left|\sum_{j=0}^{\L_{H,n,i-1}-1}\bar F(u_{n,j})-\frac{\L_{H,n,i-1}}{n}\tau\right|\to 0\]
then, by \eqref{eq:blocks},
\[\frac{\tau}{n}\sum_{i=i_\ell}^{i_\ell+\mathscr S_\ell-1}\ell_{H,n,i}\sim\frac{\tau}{n}.n\left|\frac{1}{\tau}I_\ell\right|=|I_\ell|\]
and
\[\left|\sum_{i=i_\ell}^{i_\ell+\mathscr S_\ell-1}\sum_{j=\L_{H,n,i-1}}^{\L_{H,n,i}-1}\bar F(u_{n,j})-|I_\ell|\right|\leq\sum_{i=i_\ell}^{i_\ell+\mathscr S_\ell-1}\left|\sum_{j=\L_{H,n,i-1}}^{\L_{H,n,i}-1}\bar F(u_{n,j})-\frac{\tau}{n}\ell_{H,n,i}\right|\to 0.\]
We conclude that
\[\E\left(\e^{-\sum_{\ell=1}^\varsigma y_\ell a_n \aa_n(v_n I_\ell)}\right)\sim\prod_{\ell=1}^\varsigma\prod_{i=i_\ell}^{i_\ell+\mathscr S_\ell-1}\E\Big(\e^{-y_\ell a_n\aa_n(J_i)}\Big)\sim\e^{-\theta\sum_{\ell=1}^\varsigma(1-\phi(y_\ell))|I_\ell|}=\E\left(\e^{-\sum_{\ell=1}^\varsigma y_\ell A(I_\ell)}\right)\]
where $A$ is a compound Poisson process of intensity $\theta$ and multiplicity d.f. $\pi$.

\end{proof}

\section{Application to sequential systems: an example of uniformly expanding map}
\label{sec:sequential}

In this section we will give a detailed analysis of the application of the general result obtained in Section~\ref{sec:setting} to a particular sequential system. It is constructed with $\beta$ transformations, although it can be generalised to other examples of sequential systems presented in \cite[Section~3]{FFV17} after making the necessary  adaptations.

Consider the family of maps on the unit circle $S^1=[0,1]$, with  the identification $0\sim1$, given by $T_{\beta}(x)=\beta x$ mod $1$ for $\beta>1+c$, with $c>0$. Note that for many such $\beta$, we have that $T_\beta(1)\neq 1$ and, by the identification $0\sim 1$, this means that $T_\beta$ as a map on $S^1$ is not continuous at $\zeta=0\sim1$. For simplicity we assume that $T_\beta(0)=0$ but consider that the orbit of $1$ is still defined to be $T_\beta(1), T_\beta^2(1),\ldots$ although, strictly speaking, $1\sim0$ should be considered a fixed point. In what follows $m$ denotes Lebesgue measure on $[0,1]$.

\begin{theorem}\label{T3}
Consider an unperturbed map $T_\beta$ corresponding to some $\beta=\beta_0>1+c$, with invariant absolutely continuous probability $\mu=\mu_\beta$. Consider a sequential system acting on the unit circle and given by $\TF_n=T_{n}\circ\cdots\circ T_1$, where $T_i=T_{\beta_i}$, for all $i=1,\ldots,n$ and $|\beta_n-\beta|\leq n^{-\xi}$ holds for some $\xi>1$. Let $X_1, X_2,\ldots$ be defined by \eqref{eq:def-stat-stoch-proc-DS}, where the observable function $\varphi$ achieves a global maximum at a chosen periodic point $\zeta$ of prime period $p$ \footnote{$T_\beta^p(\zeta)=\zeta$ and $p$ is the minimum integer with such property} (we allow $\varphi(\zeta)=+\infty$), being of following form:
\begin{equation}
\label{eq:observable-form}
\varphi(x)=g\big(\dist(x,\zeta)\big),
\end{equation} where the function $g:[0,+\infty)\rightarrow {\mathbb
R\cup\{+\infty\}}$ is such that $0$ is a global maximum ($g(0)$ may
be $+\infty$); is a strictly decreasing homeomorphism $g:V \to W$
in a neighbourhood $V$ of $0$; and has one of the following three types of behaviour:

\begin{enumerate}
\item [Type 1:] there exists some strictly positive function $h:W\to\R$ such that for all $y\in\R$
\begin{equation}
\label{eq:def-g1}
\lim_{s\to g(0)}\frac{g^{-1}(s+yh(s))}{g^{-1}(s)}=\e^{-y};
\end{equation}
\item [Type 2:] $g(0)=+\infty$ and there exists $\beta>0$ such that for all $y>0$
\begin{equation}
\label{eq:def-g2}
\lim_{s\to+\infty}\frac{g^{-1}(sy)}{g^{-1}(s)}=y^{-\beta};\end{equation}
\item [Type 3:] $g(0)=D<+\infty$ and there exists $\gamma>0$ such that for all $y>0$
\begin{equation}
\label{eq:def-g3}
\lim_{s\to 0}\frac{g^{-1}(D-sy)}{g^{-1}(D-s)}=y^\gamma.
\end{equation}
\end{enumerate}

Let $(u_n)_{n\in\N}$ be such that $n\mu(X_0>u_n)\to\tau$, as $n\to\infty$ for some $\tau\geq 0$.

Then, the POT MREPP $a_nA_n$ converges in distribution to a compound Poisson process with intensity $\theta$ given by
\begin{equation}
\theta=\begin{cases}
1-\beta^{-p}, \text{when the orbit of $\zeta$ by $T_\beta$ never hits $0\sim1$}\\
\frac{d\mu}{dm}(0)(1-\beta^{-1})+\frac{d\mu}{dm}(1)(1-\beta^{-p}), \text{when $\zeta=0\sim1$}
\end{cases}
\end{equation}
and multiplicity distribution
\begin{equation}
\label{eq:POT-multiplicity}
\pi(x)=\begin{cases}
1-\e^{-x}, \text{when $g$ is of type 1 and $a_n=h(u_n)^{-1}$}\\
1-(1+x)^{-\beta}, \text{when $g$ is of type 2 and $a_n=u_n^{-1}$}\\
1-(1-x)^{\gamma}, \text{when $g$ is of type 3 and $a_n=(D-u_n)^{-1}$}
\end{cases}
\end{equation}

and, for $a_n=h(u_n)^{-1}$ the AOT MREPP $a_nA_n$ converges in distribution to a compound Poisson process with the same intensity $\theta$ as above and multiplicity d.f. $\pi$ given by
\begin{equation}
\label{eq:AOT-multiplicity}
\pi(x)=1-\lim_{n\to\infty}h_{\kappa(u_n,q(u_n)x)}(x)
\end{equation}
where $g_{\kappa,u}(x)=\sum_{i=0}^\kappa(g(M^ix)-u)$ with $M=\beta^p$; $\kappa=\kappa(u,x)$ is the only integer such that $x\in\left[g_{\kappa,u}\left(\frac{g^{-1}(u)}{M^\kappa}\right),g_{\kappa,u}\left(\frac{g^{-1}(u)}{M^{\kappa+1}}\right)\right)$; and $h_k$ is a strictly monotone homeomorphism $h_k$ such that
\begin{equation}
\label{eq:def-h}
\lim_{u\to g(0)}\frac{g_{\kappa,u}(g^{-1}(u)h_\kappa(x))}{h(u)}=x.
\end{equation}

\end{theorem}

\begin{remark}
Examples of each one of the three types are as follows:
$g(x)=-\log x$ (in this case \eqref{eq:def-g1} is easily verified
with $h\equiv 1$), $g(x)=x^{-1/\alpha}$ for some $\alpha>0$ (condition \eqref{eq:def-g2} is verified with $\beta=\alpha$) and
$g(x)=D-x^{1/\alpha}$ for some $D\in\R$ and $\alpha>0$ (condition
\eqref{eq:def-g3} is verified with $\gamma=\alpha$). For these examples, the multiplicity d.f. of the compound Poisson process associated to the AOT MREPP $a_nA_n$ can be computed as shown in the following table:

\begin{tabular}{c|p{11cm}}
 Examples of $g(x)$ & Respective distribution $\pi(x)$ \\
\hline
\rule{0pt}{5ex}
 $-\log(x)$ & $1-(\sqrt{M})^{-\lfloor\frac{\sqrt{1+8x/\log M}-1}{2}\rfloor}\e^{-\frac{x}{\lfloor\frac{\sqrt{1+8x/\log M}-1}{2}\rfloor+1}}$ \\
\rule{0pt}{6ex}
 $x^{-1/\alpha}$ & $1-\left(\frac{1-M^{-1/\alpha}}{1-M^{-(\kappa(x)+1)/\alpha}}\right)^{-\alpha}(\kappa(x)+1+x)^{-\alpha}$
where $\kappa=\kappa(x)$ is the only integer such that $\frac{M^{\kappa/\alpha}-M^{-1/\alpha}}{1-M^{-1/\alpha}}\leq \kappa+1+x<\frac{M^{(\kappa+1)/\alpha}-1}{1-M^{-1/\alpha}}$ \\
\rule{0pt}{6ex}
 $D-x^{1/\alpha}$ & $1-\left(\frac{1-M^{1/\alpha}}{1-M^{(\kappa(x)+1)/\alpha}}\right)^\alpha(\kappa(x)+1-x)^\alpha$ where $\kappa=\kappa(x)$ is the only integer such that $\frac{1-M^{-(\kappa+1)/\alpha}}{M^{1/\alpha}-1}<\kappa+1-x\leq\frac{M^{1/\alpha}-M^{-\kappa/\alpha}}{M^{1/\alpha}-1}$\\
\end{tabular}
\end{remark}

\begin{remark}
We point out that in this example we take $u_{n,i}=u_n$, where $(u_n)_{n\in\N}$ satisfies $n\mu(X_0>u_n)\to\tau$, as $n\to\infty$ for some $\tau>0$, where $\mu$ is the invariant measure of the original map $T_\beta$.\\
\end{remark}

\subsection{Preliminaries}
As we said above, we let $\mu$ denote the invariant measure of the original map $T_\beta$ and let $h=\frac{d\mu}{dm}$ be its density. In what follows, let $U_n=\{X_0>u_n\}$.

We will assume throughout this subsection the existance of some $\xi>1$ such that
\begin{equation}
\label{eq:beta-rate}
|\beta_{n}-\beta|\leq \frac1{n^\xi}.
\end{equation}
Also let $0<\gamma<1$ be such that $\gamma\xi>1$. In what follows $P$ denotes the Perron-Fr\"obenius transfer operator associated to the unperturbed map $T_\beta$ with respect to the reference Lebesgue measure $m$. Recall that $\Pi_i=P_i\circ\ldots\circ P_1$, where $P_i$ is the transfer operator associated to $T_i=T_{\beta_i}$, while $P^i$ is the corresponding concatenation for the unperturbed map $T_\beta$. Note that by \cite[Lemma~3.10]{CR07}, we have
\begin{equation}
\label{eq:PF-approximation}
\left\| \Pi_i(g)-\int g dm\; h\right\|_1\leq C_1\frac{\log i}{i^\xi}\|g\|_{BV}.
\end{equation}

For any measurable set $A\subset [0,1]$, we have
\begin{align*}
m(\TF_i^{-1}(A))&=\int \I_A\circ T_i\circ\ldots\circ T_1 dm =\int  \I_A \Pi_i(1) dm\\
&=\int \I_{A} hdm + \int \I_{A}(\Pi_i(1)-h)dm.
\end{align*}
By \eqref{eq:PF-approximation}, if $i\geq\lfloor n^\gamma\rfloor$ (recall that $\gamma\xi>1$) then we have $\int|\Pi_i(1)-h| dm\leq C_1 \frac{\log i}{i^\xi}=o(n^{-1}),$ which allows us to write:
\begin{equation}
\label{eq:iterated-measure}
m(\TF_i^{-1}(A))=\mu(A)+o(n^{-1}).
\end{equation}

\subsubsection{Verification of condition \eqref{F_hn}} We want to show that $\sum_{i=0}^{h_n-1}m(X_i>u_n)=\frac{h_n}{n}\tau+o(1)$ for any unbounded increasing sequence of positive integers $h_n\leq Hn$.

We begin with the following lemma.
\begin{lemma}
\label{lem:m-stationarity}
We have that
$$\sum_{i=0}^{h_n-1}\int_{U_n} P^i(1) \,dm=\frac{h_n}{n}\tau+o(1).$$
\end{lemma}
\begin{proof}
By hypothesis, for all $i\in\N$ and $g\in BV$ we have $P^i(g)=h\int g \cdot h \,dm + Q^i(g)$, where $\|Q^i(g)\|_\infty\leq \alpha^i \|g\|_{BV}$, for some $\alpha<1$. Then we can write:
\begin{align*}
\sum_{i=0}^{h_n-1}\int_{U_n} P^i(1)dm&=\sum_{i=0}^{h_n-1}\int h \left(\int 1\cdot h dm\right) \1_{U_n} dm+\sum_{i=0}^{h_n-1} \int Q^i(1)\1_{U_n} dm\\
&=\sum_{i=0}^{h_n-1}\int_{U_n} h dm +\sum_{i=0}^{h_n-1} \int Q^i(1)\1_{U_n} dm\\
&=\frac{h_n}{n} n\mu(U_n)+\sum_{i=0}^{h_n-1} \int Q^i(1)\1_{U_n} dm.
\end{align*}
The result will follow once we show that the second term on the right goes to 0, as $n\to\infty$. This follows easily because
$$
\sum_{i=0}^{h_n-1} \int Q^i(1)\1_{U_n} dm\leq \sum_{i=0}^{h_n-1} \alpha^i \int \1_{U_n} dm =\frac{1-\alpha^{h_n}}{1-\alpha}m(U_n)\xrightarrow[n\to\infty]{}0.
$$
\end{proof}

Since
$$
\sum_{i=0}^{h_n-1}m(X_i>u_n)=\sum_{i=0}^{h_n-1}\int_{U_n} \Pi_i(1)dm=\sum_{i=0}^{h_n-1}\int_{U_n} P^i(1)dm+ \sum_{i=0}^{h_n-1}\int_{U_n} \Pi_i(1)-P^i(1)dm,
$$
then condition \eqref{F_n} holds once we prove that the second term on the right goes to 0 as $n\to\infty$.

Let $\varepsilon>0$ be arbitrary. Since $\xi>1$ then $\sum_{i\geq0} \frac{\log i}{i^\xi}<\infty$, so there exists $N\geq\lfloor n^\gamma\rfloor$ such that $C_1\sum_{i\geq N} \frac{\log i}{i^\xi}<\eps/2$.

On the other hand, using the Lasota-Yorke inequalities (see \cite[Section~3]{FFV17}) for both $\Pi$ and $P$, we have that there exists some $C>0$ such that $|\Pi_i(1)-P^i(1)|\leq C$, for all $i\in\N$. Let $n$ be sufficiently large so that $C N m(U_n)<\eps/2$. Then
\begin{align*}
\left|\sum_{i=0}^{h_n-1}\int_{U_n}\Pi_i(1)-P^i(1)dm\right|&\leq \sum_{i=0}^{N-1}\int_{U_n}|\Pi_i(1)-P^i(1)|dm+\sum_{i=N}^{\infty}\int_{U_n}|\Pi_i(1)-P^i(1)|dm\\
&\leq C N m(U_n)+C_1\sum_{i\geq N} \frac{\log i}{i^\xi}<\eps/2+\eps/2=\eps.
\end{align*}

\subsection{Verification of condition $\D_q(u_{n,i})^*$}
\label{subsec:D-check}
We will use the following proposition, proved in \cite[Section~3]{FFV17}.
\begin{proposition}
\label{prop:decay-correlations}
Let $\phi\in BV$ and $\psi\in L^1(m)$. Then for the $\beta$ transformations $T_n=T_{\beta_n}$ we have that
\[
\left|\int \phi\circ \TF_i\psi \circ \TF_{i+t}dm-\int \phi\circ \TF_i dm\int  \psi \circ \TF_{i+t} dm \right|\leq \mbox{B} \lambda ^t\|\phi\|_{BV}\|\psi\|_{1},
\]
for some $\lambda<1$ and $\mbox{B}>0$ independent of $\phi$ and $\psi$.
\end{proposition}

\begin{remark}
\label{rem:dec-corr}
As it can be seen in \cite[Section~3]{CR07}, Proposition \ref{prop:decay-correlations} holds for any sequence $T_{\beta_1}, T_{\beta_2}, \ldots$ of $\beta$ transformations and not necessarily only for the ones that satisfy condition \eqref{eq:beta-rate}.
\end{remark}

Condition $\D_q(u_{n,i})^*$ follows from Proposition~\ref{prop:decay-correlations} by taking for each $i\leq H n-1$,
\begin{equation*}
\phi_i=\I_{D_{n,i}(x_1)} \mbox{ and }\psi_i=\I_{\bigcap_{j=2}^\varsigma\{\aa_n(I_j-i-t)\leq x_j\}},
\end{equation*}
where for every $j\leq H n-1$ we define
\begin{equation}
D_{n,j}(x):=B_{n,0}(x)\cap \bigcap_{\ell=1}^q (T_{j+\ell}\circ\ldots\circ T_{j+1})^{-1}(B_{n,0}(x))^c.
\end{equation}
Since we assume that \eqref{eq:beta-rate} holds, there exists a constant $C>0$ depending on $x_1$ but not on $i$ such that $\|\phi_i\|_{BV}<C$. Moreover, it is clear that $\|\psi_i\|_1\leq 1$. Hence,
\begin{multline*}
\left|m\left(A_{n,i}(x_1)\cap \bigcap_{j=2}^\varsigma\{\aa_n(I_j)\leq x_j\}\right)-m\left(A_{n,i}(x_1)\right)m\left(\bigcap_{j=2}^\varsigma\{\aa_n(I_j)\leq x_j\}\right)\right|\\
=\left|\int \phi_i\circ \TF_i\psi_i \circ \TF_{i+t}dm-\int \phi_i\circ \TF_i dm\int \psi_i \circ \TF_{i+t} dm \right|\leq \mbox{const}\; \lambda^t.
\end{multline*}
Thus, if we take $\gamma_i(n,t)=\mbox{const} \lambda^t$ and $t_n^*=(\log n)^2$, condition $\D_q(u_{n,i})^*$ is trivially satisfied.

\subsection{Verification of condition $\D'_q(u_{n,i})^*$}
We start by noting that we may neglect the first $\lfloor n^\gamma\rfloor$ random variables of the process $X_0, X_1, \ldots$, where $\gamma$ is such that $\gamma\xi>1$, for $\xi$ given as in \eqref{eq:beta-rate}.

In fact, by the Uniform Doeblin-Fortet-Lasota-Yorke inequality (DFLY) used in \cite[Section~3]{FFV17}, we have
\begin{align*}
m(\aa_n([\lfloor n^\gamma\rfloor,n))\leq x)-m(\aa_n([0,n))\leq x)&=m(\aa_n([0,\lfloor n^\gamma\rfloor))>0)\leq \sum_{i=0}^{\lfloor n^\gamma\rfloor-1} m(X_i>u_n)\\
&=\sum_{i=0}^{\lfloor n^\gamma\rfloor-1}\int\I_{U_n}\Pi_i(1) dm\leq C_0 n^{\gamma}m(U_n)\xrightarrow[n\to\infty]{}0.
\end{align*}
This way, we simply disregard the $\lfloor n^\gamma\rfloor$ random variables of $X_0, X_1, \ldots$ and start the blocking procedure, described in Section~\ref{subsec:blocks}, in $X_{\lfloor n^\gamma\rfloor}$ by taking $\L_{H,n,0}=\lfloor n^\gamma\rfloor$. We split the remaining $n-\lfloor n^\gamma\rfloor$ random variables into $k_n$ blocks as described in Section~\ref{subsec:blocks}. Our goal is to show that
$$
S_n':=\sum_{i=1}^{k_n}\sum_{j=\L_{H,n,i-1}}^{\L_{H,n,i}-1} \sum_{r>j}^{\L_{H,n,i}-1}m\left(Q^{(0)}_{q,n,j}\cap\{X_r>u_n\}\right)+\sum_{j=\L_{H,n,k_n}}^{Hn-1} \sum_{r>j}^{Hn-1}m\left(Q^{(0)}_{q,n,j}\cap\{X_r>u_n\}\right)
$$
goes to $0$.

We define for some $j,n,q\in\N_0$,
\begin{align*}
&R_{q,n,j}:=\min\left\{r\in\N:\; Q^{(0)}_{q,n,j}\cap\{X_{j+r}>u_n\}\neq\emptyset\right\},\\
&\tilde R_{q,n}:=\min\{R_{q,n,j},\; j=\lfloor n^\gamma\rfloor, \ldots,Hn-1\},\\
&L_n:=\max\{\ell_{H,n,i},\; i=1,\ldots,k_n\}\\
&\tilde L_n:=\max\{L_n,Hn-\L_{H,n,k_n}\}.
\end{align*}

We have
$$S_n'\leq \sum_{j=\lfloor n^\gamma\rfloor}^{Hn-1}\sum_{r\geq R_{q,n,j}}^{\tilde L_n} m\left(Q^{(0)}_{q,n,j}\cap\{X_{j+r}>u_n\}\}\right)= \sum_{j=\lfloor n^\gamma\rfloor}^{Hn-1}\sum_{r\geq R_{q,n,i}}^{\tilde L_n} \int \I_{D_{q,n,j}}\circ \TF_j \cdot \I_{U_n}\circ \TF_{j+r}\, dm,$$
where for every $j\leq H n-1$ we define
\begin{equation}
D_{q,n,j}:=U_n\cap\bigcap_{\ell=1}^q (T_{j+\ell}\circ\ldots\circ T_{j+1})^{-1}(U_n)^c.
\end{equation}
Using Proposition~\ref{prop:decay-correlations}, with $\phi=\I_{D_{q,n,j}}$ and $\psi=\I_{U_n}$, and the adjoint property of the operators, it follows that
$$
\int \I_{D_{q,n,j}}\circ \TF_j \cdot \I_{U_n}\circ \TF_{j+r}\, dm\leq \int \I_{D_{q,n,j}} \Pi_j(1)dm \int \I_{U_n} \Pi_{j+r}(1)dm+B \lambda^r\|\I_{D_{q,n,j}}\|_{BV}\|\I_{U_n}\|_1.
$$
Using (DFLY), we have
$$
\int \I_{D_{q,n,j}}\circ \TF_j \cdot \I_{U_n}\circ \TF_{j+r}\, dm\leq C_0^2 m(U_n)^2+B C_2\lambda^r m(U_n)
$$
for some $C_2>0$ (independent of $n$) such that $\|\I_{D_{q,n,j}}\|_{BV}\leq C_2$. Hence,
\begin{align*}
S_n'&\leq \sum_{j=\lfloor n^\gamma\rfloor}^{Hn-1}\sum_{r\geq R_{q,n,i}}^{\tilde L_n} \left(C_0^2 m(U_n)^2+B C_2\lambda^r m(U_n)\right)\leq C_0^2 Hn\tilde L_n m(U_n)^2+B C_2 m(U_n)Hn\sum_{r\geq \tilde R_{q,n}}^{\tilde L_n} \lambda^r\\
&\leq C_0^2 Hn\tilde L_n m(U_n)^2+BC_2 m(U_n)Hn\lambda^{\tilde R_{q,n}}\frac{1}{1-\lambda}.
\end{align*}
Now we show that $\tilde L_n=o(n)$. To see this, observe that each $\ell_{H,n,i}$ is defined, in this case, by the largest integer $\ell$ such that
$$\sum_{j=\L_{H,n,i-1}}^{\L_{H,n,i-1}+\ell-1}m(X_j>u_n)\leq\frac{1}{k_n}\sum_{j=\lfloor n^\gamma\rfloor}^{Hn-1}m(X_j>u_n).$$
Using \eqref{eq:iterated-measure}, it follows that
$$\ell_{H,n,i}\mu(U_n)(1+o(1))\leq\frac{Hn-\lfloor n^\gamma\rfloor}{k_n}\mu(U_n)(1+o(1)).$$
On the other hand, by definition of $\ell_{H,n,i}$ we must have
$$\sum_{j=\L_{H,n,i-1}}^{\L_{H,n,i-1}+\ell_{H,n,i}-1}m(X_j>U_n)>\frac{1}{k_n}\sum_{j=\lfloor n^\gamma\rfloor}^{Hn-1}m(X_j>u_n)-m(X_{\L_{H,n,i-1}+\ell_{H,n,i}}>u_n).$$
Using \eqref{eq:iterated-measure} again, we have
$$\ell_{H,n,i}\mu(U_n)(1+o(1))>\frac{Hn-\lfloor n^\gamma\rfloor}{k_n}\mu(U_n)(1+o(1))-\mu(U_n)(1+o(1)).$$
Together with the previous inequality, we have
\begin{equation}
\label{eq:ln-estimate}
\ell_{H,n,i}=\frac{Hn-\lfloor n^\gamma\rfloor}{k_n}(1+o(1))=o(n)
\end{equation}
for every $i=1,\ldots,k_n$ and
$$Hn-\L_{H,n,k_n}=Hn-\lfloor n^\gamma\rfloor-\sum_{i=1,\ldots,k_n}\ell_{H,n,i}=(Hn-\lfloor n^\gamma\rfloor)o(1)=o(n)$$
so $\tilde L_n=o(n)$ follows at once. Using this estimate, the fact that $\lim_{n\to\infty}n\mu(U_n)=\tau$ and $h\in BV$, we have $C_0^2 Hn\tilde L_n m(U_n)^2\to 0$.

In order to prove that $\D'_q(u_{n,i})^*$ holds, we need to show that $\tilde R_{q,n}\to\infty$, as $n\to\infty$.
To do that we consider two cases, whether the orbit of $\zeta$ hits $1$ or not.

We will consider that the maps $T_i$, for all $i\in\N_0$, are defined in $S^1$ by using the usual identification $0\sim 1$. Observe that the only possible point of discontinuity of such maps is $0\sim 1$. Moreover, $\lim_{x\to 0^+}T_i(x)=0$ and $\lim_{x\to1^-}T_i(x)=\beta_i-\lfloor \beta_i\rfloor.$

\subsubsection{The orbit of $\zeta$ by the unperturbed $T_\beta$ map does not hit 1} We mean that for all $j\in\N_0$ we have $T^j(\zeta)\neq 1$.

We take $q=p$, where $p\in\N$ is such that $T^p(\zeta)=\zeta$ and $T^j(\zeta)\neq\zeta$ for all $j<p$. Let
\begin{equation}
\label{eq:def-eps-n}
\eps_n:=|\beta_{\lfloor n^\gamma\rfloor}-\beta|.
\end{equation}
By \eqref{eq:beta-rate} and choice of $\gamma$, we have that $\eps_n=o(n^{-1})$. Also let $\delta>0$, be such that $B_\delta(\zeta)$ is contained on a domain of injectivity of all $T_i$, with $i\geq\lfloor n^\gamma\rfloor$.

Let $J\in\N$ be chosen. Using a continuity argument, we can show that there exists $C:=C(J,q)>0$ such that
$$
\dist(T_{i+j}\circ\ldots\circ T_{i+1} (\zeta), T^j(\zeta))< C\eps_n, \mbox{ for all $j=1, \ldots, J$}
$$
and moreover $U_n\cap T_{i+j}\circ\ldots\circ T_{i+1}(U_n)=\emptyset$, for all $j\leq J$ such that $j/q-\lfloor j/q\rfloor>0$.

We want to check that if $x\in Q^{(0)}_{q,n,i}$ for some $i\geq\lfloor n^\gamma\rfloor$, \ie $\TF_i(x)\in D_{q,n,i}$, then $X_{i+j}(x)\leq u_n$, for all $j=1,\ldots,J$. By the assumptions above, we only need to check the latter for all $j=1, \ldots, J$ such that $j/q-\lfloor j/q\rfloor=0$, \ie for all $j=sq$, where $s=1, \ldots, \lfloor J/q\rfloor$.

By definition of $Q^{(0)}_{q,n,i}$ the statement is clearly true when $s=1$. Now, we consider $s>1$ and let $x\in Q^{(0)}_{q,n,i}$. We have
$$
\dist(\TF_{i+sq}(x), T_{i+sq}\circ\ldots\circ T_{i+q+1}(\zeta))>(\beta-\eps_n)^{(s-1)q}\dist(\TF_{i+q}(x),\zeta).
$$
On the other hand,
$$
\dist(T_{i+sq}\circ\ldots\circ T_{i+q+1}(\zeta), \zeta)\leq C\eps_n.
$$
Hence,
\begin{align*}
\dist(\TF_{i+sq}(x), \zeta)&\geq \dist(\TF_{i+sq}(x), T_{i+sq}\circ\ldots\circ T_{i+q+1}(\zeta))- \dist(T_{i+sq}\circ\ldots\circ T_{i+q+1}(\zeta), \zeta)\\
&\geq (\beta-\eps_n)^{(s-1)q}\dist(\TF_{i+q}(x),\zeta)-C\eps_n\\
&\geq (\beta-\eps_n)^{(s-1)q}\frac{m(U_n)}{2}-C\eps_n, \mbox{ since $x\in Q^{(0)}_{q,n,i}\Rightarrow X_{i+q}(x)\leq u_n \Leftrightarrow\TF_{i+q}(x)\notin U_n$}\\
&>\frac{m(U_n)}{2}, \mbox{ for $n$ sufficiently large, since $\eps_n=o(n^{-1})$.}
\end{align*}
This shows that $\TF_{i+sq}(x)\notin U_n$, which means that $X_{i+sq}(x)\leq u_n$.

\subsubsection{$\zeta=0\sim 1$} In this case we proceed in the same way as in \cite[Section~3.3]{AFV15}, which basically corresponds to considering two versions of the same point: $\zeta^+=0$ and $\zeta^-=1$. Note that $\zeta^+$ is a fixed point for all maps considered and $\zeta^-$ is periodic of prime period $p$.

As the previous case, we take $q=p$. We observe that $D_{q,n,i}$ has two connected components, one to the right of 0 and the other to the left of 1, where none of the two points belongs to the set. Let $J\in\N$ be fixed as before. A continuity argument as the one used before allows us to show that the points of the components of $D_{q,n,i}$ do not return to $U_n$ before $J$ iterates, also. Note that, the maps are orientation preserving so there is no switching as described in  \cite[Section~3.3]{AFV15}.

\subsection{Verification of condition \eqref{eq:EIH}} Similarly to the previous condition, we disregard the first $\lfloor n^\gamma\rfloor$ random variables of $X_0, X_1, \ldots$ and start the blocking procedure in $X_{\lfloor n^\gamma\rfloor}$ by taking $\L_{H,n,0}=\lfloor n^\gamma\rfloor$. We want to show that
\[\lim_{n\to\infty}\max_{i=1,\ldots,k_n}\left\{\left|\theta\sum_{j=\L_{H,n,i-1}}^{\L_{H,n,i}-1}m(X_j>u_n)-\sum_{j=\L_{H,n,i-1}}^{\L_{H,n,i}-1}m\left(Q_{q,n,j}^{(0)}\right)\right|\right\}=0.\]

Let $\eps_n$ be defined as in \eqref{eq:def-eps-n} and let $\delta_n$ be such that $U_n=B_{\delta_n}(\zeta)$. For simplicity, we assume that we are using the usual Riemannian metric so that we have a symmetry of the balls, which means that $|U_n|=m(U_n)=2\delta_n$.

We also assume that $\zeta$ is a periodic point of prime period $p$ with respect to the unperturbed map $T=T_\beta$ and the orbit of $\zeta$ does not hit $0\sim 1$. In this case, we take $\theta=1-\beta^{-q}$ with $q=p$ and check \eqref{eq:EIH}.

Using a continuity argument we can show that there exists $C:=C(J,q)>0$ such that
$$
\dist(T_{i+q}\circ\ldots\circ T_{i+1} (\zeta), \zeta)< C\eps_n.
$$
We define two points $\xi_u$ and $\xi_l$ of $B_{\delta_n}(\zeta)$ on the same side with respect to $\zeta$ such that $\dist(\xi_u,\zeta)=(\beta-\eps_n)^{-q}\delta_n+C\eps_n$ and $\dist(\xi_l,\zeta)=(\beta+\eps_n)^{-q}\delta_n-(\beta+\eps_n)^{-q}C\eps_n$. Recall that for all $i\geq\lfloor n^\gamma\rfloor$, we have that $(\beta-\eps_n)^q\leq \beta_{i+1}\cdot\ldots\cdot\beta_{i+q}\leq (\beta+\eps_n)^q$.

Since we are composing $\beta$ transformations, then for all $i\geq\lfloor n^\gamma\rfloor$, we have
$$
\dist(T_{i+q}\circ\ldots\circ T_i(\xi_u), T_{i+q}\circ\ldots\circ T_i(\zeta))\geq \delta_n+(\beta-\eps_n)^q C\eps_n.
$$
Using the triangle inequality it follows that
$$
\dist(T_{i+q}\circ\ldots\circ T_{i+1}(\xi_u), \zeta)\geq \delta_n.
$$
Similarly, $\dist(T_{i+q}\circ\ldots\circ T_{i+1}(\xi_l), T_{i+q}\circ\ldots\circ T_{i+1}(\zeta))\leq \delta_n-C\eps_n$ and
$$
\dist(T_{i+q}\circ\ldots\circ T_{i+1}(\xi_l), \zeta)\leq \delta_n.
$$
If we assume that both $\xi_u$ and $\xi_l$ are on the right hand side with respect to $\zeta$ and $\xi_u^*$ and $\xi_l^*$ are the corresponding points on the left hand side of $\zeta$, then
$$
(\zeta-\delta_n,\xi_u^*]\cup[\xi_u, \zeta+\delta_n)\subset D_{q,n,i}\subset (\zeta-\delta_n,\xi_l^*]\cup[\xi_l, \zeta+\delta_n).
$$
Hence,
$$
\delta_n-(\beta-\eps_n)^{-q}\delta_n-C\eps_n\leq \frac{1}{2}m(D_{q,n,i})\leq \delta_n-(\beta+\eps_n)^{-q}\delta_n+(\beta+\eps_n)^{-q}C\eps_n.
$$
Since $\eps_n=o(n^{-1})=o(\delta_n)$ then we easily get
\begin{equation*}
\lim_{n\to\infty}\frac{m(D_{q,n,i})}{m(U_n)}=1-\beta^{-q}.
\end{equation*}

Observe that by \eqref{eq:iterated-measure}, $m(Q^{(0)}_{q,n,i})=m(\TF_i^{-1}(D_{q,n,i}))=\mu(D_{q,n,i})+o(n^{-1})$ and $m(X_i>u_n)=\mu(U_n)+o(n^{-1})$. Hence, we have that
\begin{equation*}
\lim_{n\to\infty}\frac{m(Q^{(0)}_{q,n,i})}{m(X_i>u_n)}=\lim_{n\to\infty}\frac{\mu(D_{q,n,i})}{\mu(U_n)}.
\end{equation*}

The density $\frac{d\mu}{dm}$, which can be found in \cite[Theorem~2]{P60}, is sufficiently regular so that, as in \cite[Section~7.3]{FFT15}, one can see that
$$
\lim_{n\to\infty}\frac{\mu(D_{q,n,i})}{\mu(U_n)}=\lim_{n\to\infty}\frac{m(D_{q,n,i})}{m(U_n)}.
$$
It follows that
$$
\lim_{n\to\infty}\frac{m(Q^{(0)}_{q,n,i})}{m(X_i>u_n)}=1-\beta^{-q}.
$$
Since, as we have seen in \eqref{eq:ln-estimate}, we can write $\ell_{H,n,i}=\frac{Hn}{k_n}(1+o(1))$, then the previous equation can easily be used to prove that condition \eqref{eq:EIH} holds, with $\theta=1-\beta^{-q}$.

For the case $\zeta=0\sim1$ the argument will follow similarly, although we have to take into account the fact that the density is discontinuous at $0\sim1$. By \cite{P60} we have that $$\frac{d\mu}{dm}(x)=\frac1{M(\beta)}\sum_{x<T^n (1)} \frac1{\beta^n},$$ where $M(\beta):=\int_0^1\sum_{x<T^n (1)} \frac1{\beta^n}dm$.  In this case, we have $\theta=\frac{d\mu}{dm}(0)(1-\beta^{-1})+\frac{d\mu}{dm}(1)(1-\beta^{-q})$.

\subsection{Verification of condition \eqref{eq:multiplicity}} Once again, we disregard the first $\lfloor n^\gamma\rfloor$ random variables of $X_0, X_1, \ldots$. We want to show that
\[\lim_{n\to\infty}\max_{j=\lfloor n^\gamma\rfloor,\ldots,Hn-1}\left\{\left|\frac{m(A_{n,j}(x/a_n))}{m\left(Q^{(0)}_{q,n,j}\right)}-(1-\pi(x))\right|\right\}=0.\]

Observing that by \eqref{eq:iterated-measure}, $m(B_{n,i}(x))=m(\TF_i^{-1}(B_{n,0}(x)))=\mu(B_{n,0}(x))+o(n^{-1})$ and using an argument similar to the one of the previous condition, we have
\begin{equation*}
\lim_{n\to\infty}\frac{m(A_{n,i}(x))}{m(B_{n,i}(x))}=\lim_{n\to\infty}\frac{\mu(D_{n,i}(x))}{\mu(B_{n,0}(x))}=\lim_{n\to\infty}\frac{m(D_{n,i}(x))}{m(B_{n,0}(x))}=\theta
\end{equation*}
where
\begin{equation}
D_{n,j}:=\TF_j^{-1}(Q^{(0)}_{q,n,j})=U_n\cap\bigcap_{\ell=1}^q (T_{j+\ell}\circ\ldots\circ T_{j+1})^{-1}(U_n)^c
\end{equation}
and with the same $\theta$ as before. Hence,
\begin{equation*}
\lim_{n\to\infty}\frac{m(A_{n,i}(x))}{m(Q^{(0)}_{q,n,i})}=\lim_{n\to\infty}\frac{\theta m(B_{n,i}(x))}{\theta m(X_i>u_n)}=\lim_{n\to\infty}\frac{m(B_{n,i}(x))}{m(X_i>u_n)}
\end{equation*}
Let $\tilde B_{n,i}(x)$ be the set $B_{n,i}(x)$ associated to the unperturbed dynamical system given by $\TF_n=(T_{\beta})^n$ and $\tilde X_i$ the corresponding unperturbed random variables. Using a continuity argument we can show that $m(B_{n,i}(x))\sim m(\tilde B_{n,i}(x))$ and $m(X_i>u_n)\sim m(\tilde X_i>u_n)$, so that
\begin{equation*}
\lim_{n\to\infty}\frac{m(B_{n,i}(x))}{m(X_i>u_n)}=\lim_{n\to\infty}\frac{m(\tilde B_{n,i}(x))}{m(\tilde X_i>u_n)}=\lim_{n\to\infty}\frac{\mu(\tilde B_{n,i}(x))}{\mu(\tilde X_i>u_n)}=\lim_{n\to\infty}\frac{\mu(\tilde B_{n,0}(x))}{\mu(U_n)}
\end{equation*}

For this unperturbed stationary process, it has been proved in \cite[Section~3]{FFM18} that $\lim_{n\to\infty}\frac{\mu(\tilde B_{n,0}(x/a_n))}{\mu(U_n)}=1-\pi(x)$, where $\pi(x)$ is the distribution given in \eqref{eq:POT-multiplicity} for the POT MREPP $a_nA_n$ and given in \eqref{eq:AOT-multiplicity} for the AOT MREPP $a_nA_n$. So, $\lim_{n\to\infty}\frac{m(A_{n,i}(x/a_n))}{m(Q^{(0)}_{q,n,i})}=1-\pi(x)$ for that same distribution $\pi(x)$ and for any $i=\lfloor n^\gamma\rfloor,\ldots,Hn-1$. Hence, \eqref{eq:multiplicity} follows at once.

\subsection{Verification of condition $U\!LC_q(u_{n,i})$} We want to see that, for all $H\in\N$ and $y>0$,
\[\lim_{n\to\infty}\;\sum_{i=1}^{k_n}\int_0^{\infty}\e^{-y x}\delta_{n,\L_{H,n,i-1},\ell_{H,n,i}}(x/a_n)dx=0,\]
\[\lim_{n\to\infty}\;\int_0^{\infty}\e^{-x}\delta_{n,\L_{H,n,k_n},Hn-\L_{H,n,k_n}}(x/a_n)dx=0,\]
\[\mbox{and}\quad\lim_{n\to\infty}\;\sum_{i=1}^{k_n}\int_0^{\infty}\e^{-y x}\delta_{n,\L_{H,n,i-1},\ell_{H,n,i}-t_{H,n,i}}(x/a_n)dx=0\]

where $a_n$ is as in \eqref{eq:multiplicity} and $\delta_{n,s,\ell}(x)$ as in \eqref{eq:delta-definition}. Then, for all $x\in\R_0^+$,
\begin{align*}
\delta_{n,s,\ell}(x)&\leq\sum_{\kappa=1}^{\lfloor\ell/q\rfloor}\sum_{j=s+\ell-\kappa q}^{s+\ell-1}m\left(Q^{(\kappa)}_{q,n,j}\right)+\sum_{j=s}^{s+\ell-1}\sum_{\kappa>\lfloor\ell/q\rfloor}m\left(Q^{(\kappa)}_{q,n,j}\right)+\sum_{j=1}^q m\left(U^{(0)}_{q,n,s+\ell-j}\right)\\
&\leq\sum_{\kappa=1}^{\infty}\sum_{j=s+\ell-\kappa q}^{s+\ell-1}m\left(Q^{(\kappa)}_{q,n,j}\right)+\sum_{j=1}^q m\left(U^{(0)}_{q,n,s+\ell-j}\right)
\end{align*}

hence for all $x\in\R_0^+$ and $y\in\R^+$, we have
\[\sum_{i=1}^{k_n}\int_0^{\infty}\e^{-y x}\delta_{n,\L_{H,n,i-1},\ell_{H,n,i}}(x/a_n)dx \leq \frac{1}{y}\sum_{i=1}^{k_n}\left(\sum_{\kappa=1}^{\infty}\sum_{j=\L_{H,n,i}-\kappa q}^{\L_{H,n,i}-1}m\left(Q^{(\kappa)}_{q,n,j}\right)+\sum_{j=1}^q m\left(U^{(0)}_{q,n,\L_{H,n,i}-j}\right)\right),\]
\[\int_0^{\infty}\e^{-x}\delta_{n,\L_{H,n,k_n},Hn-\L_{H,n,k_n}}(x/a_n)dx \leq \sum_{\kappa=1}^{\infty}\sum_{j=Hn-\kappa q}^{Hn-1}m\left(Q^{(\kappa)}_{q,n,j}\right)+\sum_{j=1}^q m\left(U^{(0)}_{q,n,Hn-j}\right)\]
\[\mbox{\quad and \quad}\sum_{i=1}^{k_n}\int_0^{\infty}\e^{-y x}\delta_{n,\L_{H,n,i-1},\ell_{H,n,i}-t_{H,n,i}}(x/a_n)dx\]
\[\qquad\leq \frac{1}{y}\sum_{i=1}^{k_n}\left(\sum_{\kappa=1}^{\infty}\sum_{j=\L_{H,n,i}-t_{H,n,i}-\kappa q}^{\L_{H,n,i}-t_{H,n,i}-1}m\left(Q^{(\kappa)}_{q,n,j}\right)+\sum_{j=1}^q m\left(U^{(0)}_{q,n,\L_{H,n,i}-t_{H,n,i}-j}\right)\right).\]

Let $\tilde Q^{(\kappa)}_{q,n,j}$ and $\tilde U^{(0)}_{q,n,j}$ be the corresponding sets $Q^{(\kappa)}_{q,n,j}$ and $U^{(0)}_{q,n,j}$ associated to the unperturbed dynamical system given by $\TF_n=(T_{\beta})^n$. Using a continuity argument we can show that $m\left(Q^{(\kappa)}_{q,n,j}\right)\sim m\left(\tilde Q^{(\kappa)}_{q,n,j}\right)$ and $m\left(U^{(0)}_{q,n,j}\right)\sim m\left(\tilde U^{(0)}_{q,n,j}\right)$. For this unperturbed stationary process, it has been proved in \cite[Section~3]{FFM18} that
\[m\left(\tilde Q^{(\kappa)}_{q,n,j}\right)\sim\theta(1-\theta)^{\kappa}m\left(\tilde U^{(0)}_{q,n,j}\right).\]
so we have $m\left(Q^{(\kappa)}_{q,n,j}\right)\sim\theta(1-\theta)^{\kappa}m\left(U^{(0)}_{q,n,j}\right)$. Additionally, using \eqref{eq:iterated-measure} (once again neglecting the first $\lfloor n^\gamma\rfloor$ random variables), $m\left(U^{(0)}_{q,n,j}\right)=m(\TF_j^{-1}(U_n))\sim\mu(U_n)\sim m(U_n)$, so $m\left(Q^{(\kappa)}_{q,n,j}\right)\sim\theta(1-\theta)^{\kappa}m(U_n)$ and, by \eqref{eq:kn-sequence},
\[\sum_{i=1}^{k_n}\left(\sum_{\kappa=1}^{\infty}\sum_{j=\L_{H,n,i}-\kappa q}^{\L_{H,n,i}-1}m\left(Q^{(\kappa)}_{q,n,j}\right)+\sum_{j=1}^q m\left(U^{(0)}_{q,n,\L_{H,n,i}-j}\right)\right)\]
\[\sim k_n\left(\sum_{\kappa=1}^{\infty}\kappa q\theta(1-\theta)^{\kappa}m(U_n)+q m(U_n)\right)=\frac{k_n q}{\theta}m(U_n)\xrightarrow[n\to\infty]{}0.\]

Similarly,
\[\sum_{\kappa=1}^{\infty}\sum_{j=Hn-\kappa q}^{Hn-1}m\left(Q^{(\kappa)}_{q,n,j}\right)+\sum_{j=1}^q m\left(U^{(0)}_{q,n,Hn-j}\right)\sim\frac{q}{\theta}m(U_n)\xrightarrow[n\to\infty]{}0\]
and
\[\sum_{i=1}^{k_n}\left(\sum_{\kappa=1}^{\infty}\sum_{j=\L_{H,n,i}-t_{H,n,i}-\kappa q}^{\L_{H,n,i}-t_{H,n,i}-1}m\left(Q^{(\kappa)}_{q,n,j}\right)+\sum_{j=1}^q m\left(U^{(0)}_{q,n,\L_{H,n,i}-t_{H,n,i}-j}\right)\right)\]
\[\sim k_n\left(\sum_{\kappa=1}^{\infty}\kappa q\theta(1-\theta)^{\kappa}m(U_n)+q m(U_n)\right)=\frac{k_n q}{\theta}m(U_n)\xrightarrow[n\to\infty]{}0.\]
\section{Random dynamical systems}
\label{sec:random}

We now give another  example of  a non-stationary system in the form of  a  {\em fibred dynamical system} constructed  by taking Lasota-Yorke maps  on the fibers; we refer in particular to the paper \cite{DFGV18}.

Let us consider the unit interval  $I=[0, 1],$ endowed with the Borel $\sigma$-algebra $\mathcal B$ and the Lebesgue measure $m.$ Furthermore, let
 $$
  \var (g)=\inf_{h=g (\text{mod} \ m)}\sup_{0=s_0<s_1<\ldots <s_n=1}\sum_{k=1}^n \lvert h(s_k)-h(s_{k-1})\rvert.
 $$
 the variation of the function $g\in L^1(m).$ We define  $BV(I,m)$ (sometimes shortened in $BV$), as the  Banach space with respect to the norm
\[
 \lVert h\rVert_{BV}=\var (h)+\lVert h\rVert_1.
 \]

For a piecewise $C^2$ function $f:[0,1]\to [0,1]$, set $\delta (f)=\esinf_{x\in [0,1]} \lvert f'\rvert$ and let $N(f)$ denote the number of intervals of monotonicity of $f$. Then let  $(\Omega, \mathcal{F}, \mathbb Q)$ be  a probability space and let  $\sigma :\Omega\rightarrow \Omega$ be an  invertible $\mathbb Q$-preserving transformation. We will assume that
$\mathbb Q$ is ergodic.
Consider now a measurable map $\omega\mapsto f_\omega$, $\omega \in \Omega$  of piecewise $C^2$ maps on $[0, 1]$ defined as above and such that the map $(\omega, x)\mapsto (\mathcal P_\omega H(\omega, \cdot))(x)$ is $\mathbb Q \times m$-measurable and moreover
\begin{equation}\label{sd}
 N:=\sup_{\omega \in \Omega} N(f_\omega)<\infty, \  \delta:=\inf_{\omega \in \Omega} \delta (f_\omega)>1, \ \mbox{ and } D:=\sup_{\omega \in \Omega}|f''_\omega|_\infty<\infty.
\end{equation}
 $H$ is any $\mathbb Q\times m$ measurable function and $H(\omega, \cdot)\in L^1(m)$ for a.e. $\omega\in \Omega;$ finally $\mathcal P_\omega$ denotes the transfer (Perron-Fr\"obenius) operator associated to $f_{\omega}.$
For each $n\in \mathbb N$ and $\omega \in \Omega$, we set
\[
f_{\omega}^n=f_{\sigma^{n-1} \omega} \circ \cdots \circ f_{\omega}.
\]
For next purposes, we need two more assumption.
\begin{itemize} \item  First we  ask that
 the following uniform covering condition holds:
for every subinterval  $J\subset I, \exists k= k(J)$ s.t. for a.e.   $\omega \in \Omega, f_\omega^{k}(J) = I.$
\item Then we require the existence of $N\in \N$ such that for each $a>0$  and any sufficiently large $n\in \mathbb N$, there is  $c>0$ such that
\begin{equation*}
\esinf  \mathcal P_\omega^{Nn} h\ge c/2 \lVert h\rVert_1, \quad \text{for every $h\in C_a$ and a.e. $\omega \in \Omega$,}
\end{equation*}
where $C_a:=\{ \phi \in BV: \phi\ge 0 \text{ and } \var(\phi)\le a\int \phi\, dm \}.$ \\This cone-type condition will guarantee that the density $h_{\omega}$ constructed below is strictly positive, namely
\begin{equation}\label{INF}
\esinf h_\omega \ge c/2, \  \text{for a.e.} \  \omega \in \Omega.
\end{equation}
\end{itemize}
The next step is to introduce the probability governing the extreme value distributions. First of all we can associate to our collection of mappings on $I,$ $f_{\omega} \colon  I \to I$, $\omega \in \Omega$ the  skew product transformation  $\tau \colon  \Omega \times I \to  \Omega \times I$  defined  by
\begin{equation}
\label{eq:tau}
\tau(\omega, x)=( \sigma \omega, f_{\omega}(x)).
\end{equation}
The preceding bunch of assumptions on the maps $f_{\omega},$ allows us to show that  there exist a unique measurable and nonnegative function $h_{\omega}\colon \Omega \times I \to \mathbb R$ with the property that
	$h_\omega:=h(\omega, \cdot) \in BV$, $\int h_\omega \, dm=1$, $\mathcal L_\omega(h_\omega)=h_{\sigma \omega}$ for a.e. $\omega \in \Omega$ and
		\begin{equation}\label{ESS}
\essup_{\omega\in\Omega} \|h_\omega\|_{BV}<\infty.
\end{equation}
If we now define  a probability measure $\mu$ on $\Omega \times I$ by
\begin{equation}\label{mu}
\mu(A \times B)=\int_{A\times B} h_{\omega} \, d(\mathbb Q\times m), \quad \text{for $A\in \mathcal F$ and $B\in \mathcal B$,}
\end{equation}
then it follows  that $\mu$ is invariant with respect to $\tau$.
Furthermore, $\mu$ is obviously absolutely continuous with respect to $\mathbb Q \times m$ and is the only measure with these properties.

Let us now consider for any $\omega \in  \Omega$ the measures $\mu_{\omega}$ on the measurable space $(I,
\mathcal B)$, defined by $d\mu_{\omega}=h_
{\omega}dm.$
We recall here two important properties of these measures.
First, the so-called {\em equivariant property}: $f^*_{\omega}\mu_{\omega}=\mu_{\sigma \omega}$.
Second, the {\em disintegration} of $\mu$ on the marginal $\mathbb{Q}:$ if $A$ is any measurable set in $\mathcal F \times  \mathcal B,$ and $A_{\omega}=\{x; (\omega, x)\in A\},$ the {\em section} at $\omega,$ then $\mu(A)=\int \mu_{\omega}(A_{\omega}) d\mathbb{Q}(\omega).$

{\em The conditional (or sample) measure $\mu_{\omega}$ will constitute the probability underlying our random processes, which we called $\mathbb{P}$ in the preceding sections.}

After this preparatory work we can now state the decay of correlations result which will be used later on.
Let $\mu_\omega$ be, as above,  the measure on $X$ given by  $d\mu_\omega=h_\omega dm$ for $\omega \in \Omega$.Then there exists $K>0$ and $\rho \in (0, 1)$  such that
 \begin{equation}\label{buzzi}
 \bigg{\lvert} \int \mathcal \phi\ \psi \circ f_{\omega}^n\, d\mu_{\omega} -\int \phi \, d\mu_\omega \cdot \int \psi \, d\mu_{\sigma^n \omega} \bigg{\rvert} \le K\rho^n
 \lVert \psi\rVert_{1} \cdot \lVert \phi \rVert_{BV} ,
\end{equation}
for $n\ge 0$, $\psi \in L^1(m)$ and $\phi \in BV(X, m)$; $\lVert \cdot\rVert_{1}$ denotes the $L^1$ norm with respect to $m.$\footnote{The result in \cite{DFGV18}, Lemma 4,  is stated in a different manner. It requires $\psi$ in $L^{\infty}(m)$. Since the density $h_{\omega}$ is in $L^{\infty}(m)$ too as an element of $BV(X, m),$ and moreover is essentially bounded uniformly in $\omega$ by (\ref{ESS}), we get the  $\lVert \cdot\rVert_{1}$ norm on the right hand side of (\ref{buzzi}).}

We now choose $\Omega=Y^{\mathbb{Z}},$ where $Y=(1, \cdots, m)$ is a finite alphabet with $m$ letters. We associate to each letter a map satisfying the requirements given above: we call them {\em random Lasota-Yorke} maps. The map $\sigma$ will therefore be the bilateral shift and $\mathbb{Q}$ any ergodic shift-invariant non-atomic ergodic probability measure, for instance, and it is the choice we do here, a Bernoulli measure with weights $p_1, \cdots, p_m.$

We now consider the process given by $X_k:= \phi \circ f_{\omega}^k, k\in \mathbb{N},$ where $f_{\omega}^k:=f_{\omega_{k}} \circ \cdots \circ f_{\omega_1},$ being $\omega_j\in Y, j=1,\cdots k,$ the first $k$ symbols of the word $\omega$. The function $\phi: I\rightarrow \mathbb{R}\cup\{\pm\}$ achieves a global maximum at $z\in I$ (we allow $\phi(z)=+\infty$), being of the following form:  $\phi(x)=g(\text{dist}(x,z),$  where $g:[0, +\infty)\rightarrow \mathbb{R}\cup\{+\infty\}$ is such that $0$ is a global maximum ($g(0)$ may be $+\infty$,) and $g$ is a strictly decreasing bijection in a neighborhood of $0.$ Finally $g$ assumes one of three types of behavior which we recalled in the statement of Theorem \ref{T3}. We now introduce the marginal measure $\mu_I$ on $I$ as: $\mu_I(B)=\int_{\Omega} \mu_{\omega}(B)d\mathbb{Q}(\omega),$ with $B$ a measurable subset of $I.$ As in  \cite{FFV17} we consider all the boundary levels equal $u_{n,i}=u_n, i=1, \cdots, n-1,$ where $u_n$ is determined by the marginal measure $\mu_I$ so that
\begin{equation}\label{MAR}
\mu_n=\inf\{u\in \mathbb{R}: \mu_I(\{x\in I: \phi(x)\le u\})\ge 1-\frac{\tau}{n}\},
\end{equation}
for some $\tau>0.$ With this choice and  by Lemma 9 of \cite{RSV14}  we have
\begin{equation}\label{SEQ}
\sum_{i=0}^{Hn-1}\mu_{\sigma^i\omega}(\{x\in I: \phi(x)>u_n\})\rightarrow \tau, \ \text{as} \ n\rightarrow \infty,
\end{equation}
which is our equation (\ref{F_hn}) for the fibred systems. From now  on we will set $U_n:=\{x\in I: \phi(x)>u_n\}$ which, by the choice of the function $g,$ is an open neighborhood of the point $z.$

Condition $\D_q(u_{n,i})^*$  with $\mathbb{P}=\mu_{\omega}$ can now be worked out easily thanks to the decay of correlations (\ref{buzzi}), which takes care of observables given by characteristic functions, see the function $\psi\in L^1(m)$ in (\ref{buzzi}). We defer for the details to the second part of Proposition 4.3 in \cite{FFV17} which is the same as in the present context.

We now go  the condition $\D'_q(u_{n,i}^*)$. We should first of all elaborate about the choice of the target point $z.$ Since we have finitely many maps $f_{k}$ each of which with finitely many branches, we could choose the point $z$ on a set of full $m$ measure in such a way that it will not intersect the preimages of any order of any of the maps $f_1,\cdots, f_m.$ We should also remember that the statement on the convergence in distribution for the extreme value law should hold for $\mathbb{Q}$-almost all choice of $\omega$ defining the sample measure $\mu_{\omega}.$  This will be useful in the following {\em periodicity} considerations, which will allow us to choose $q=0$ in the conditions  $\D_q(u_{n,i})^*$ and $\D'_q(u_{n,i}^*)$  above. We begin to notice that three situations can occur:
\begin{itemize}
\item For a given $\omega$, the point $z$ will never come  back to itself, namely $f^k_{\omega}z\neq z, \forall k\ge 1.$
\item For a given $\omega$ there are finitely many blocks of periodicity, namely we have finitely many sequences of type $\omega_{i_1}\cdots \omega_{i_L}\in \omega$ for which $f_{\omega_{i_L}}\cdots f_{\omega_{i_1}}z=z.$
\item For a given $\omega$ there are countably  many blocks of periodicity like those described in the preceding item.
\end{itemize}
We  begin to observe that the set of the words with infinitely many blocks of periodicity has measure zero. We therefore treat now the words with finitely many blocks of periodicity, the  situation in the first item being included in that one.
Having fixed such an $\omega$, call $n_{\omega}$ the last time $f_{\omega}^{n_{\omega}}z=z.$ The proof follows now closely that in section 4.3.1 on the \cite{FFV17} paper to which we defer for the details. We now point out the main differences arising in our framework.
\begin{itemize}
\item First of all we use the quenched decay of correlations established in (\ref{buzzi}) applied to the same observable ${\bf 1}_{U_n}.$ This will produce two asymptotic terms $\mu_{\sigma^i\omega}(U_n)$ and $\mu_{\sigma^j\omega}(U_n), j>i,$ and the exponential error term containing the Lebesgue measure $m(U_n)$.
\item The measure of $\mu_{\sigma^i\omega}(U_n)$ will appear in a sum ranging from $1$ to $n$ and therefore it will converge to $\tau$ by (\ref{SEQ}).
    \item The other measure should be expressed in terms of the Lebesgue measure $m$ in order to compare it with the error term and to establish bounds from below and from above for the quantity $\mathcal{L}_{H,n}:=\max \{\mathcal{L}_{H,n,i}, i=1,\cdots, k_n\}.$ Thanks to (\ref{INF}) and (\ref{ESS}), we have that there exists two constants $c_1$ and $c_2$ such that for $\mathbb{Q}$-almost any $\omega \in \Omega$ we have that
        $$
        c_1 m(U_n)\le \mu_{\sigma^i\omega}(U_n)\le c_2 m(U_n), \ \forall i\ge 1.
        $$
        \item We now come to the main difference with the analogous proof in section 4.3.1 in \cite{FFV17}. We have to prove that if $f^i_{\omega}(x) \in U_n,$ then $f^j_{\omega}(x) \in U_n,$ for the next time with $j$ growing to infinity. We already put $n_{\omega}$ the last time $f_{\omega}^{n_{\omega}}z=z.$  If we now fix $J \in \mathbb{N},$ then $f_{\omega}^{n_{\omega}+k}z, \  k=1, \cdots, J,$ will never return to $z.$ Since we are composing finitely many maps, there will be an $\varepsilon>0$, such that $\forall \omega \in \Omega$ and $k=1, \cdots, J$ we have $\text{dist}(f_{\omega}^{n_{\omega+k}}z, z)>\varepsilon.$ Call $\overline{n}$ the integer such that $\text{diameter}(U_{\overline{n}})<\frac{\varepsilon}{4}\delta^{-J}$ and $U_{\overline{n}}$ does not intersect the preimages up to order $J$ of the family of maps $f_k, k=1,\cdots, J.$ If we now take $n>\max\{n_{\omega}, \overline{n}\},$ we have that  $\forall x\in U_n$
$\text{dist}(f_{\omega}^J x, z)>\frac{\varepsilon}{2}.$
\end{itemize}
We are left with the verification of conditions \ref{eq:EIH} and \ref{eq:multiplicity}. By (\ref{SEQ}) and the definition of $Q^{(0)}_{0,n,j}= \{\phi\circ f_{\omega}^j>u_n\},$ we see immediately that $\theta=1.$ The computation of \ref{eq:multiplicity} follows closely that in the proof of Theorem 3.A in \cite{FFM18}; we give the details for the type-1 observable $g=-\log x,$ for which $h=1.$  We are reduced to estimate the ratio $\frac{\mu_{\sigma^j}(X_0>u_n+x)}{\mu_{\sigma^j}(X_0>u_n)}$, where $X_0(\cdot)=-\log \text{dist}(\cdot, z)$ and $z$ is chosen $m$-almost everywhere. We have
$$
\frac{\mu_{\sigma^j}(X_0>u_n+x)}{\mu_{\sigma^j}(X_0>u_n)}=
\frac{m(B(z, e^{-u_n-x}))}{m(B(z, e^{-u_n}))}\frac{m(B(z, e^{-u_n}))\int_{B(z, e^{-u_n-x})}h_{\sigma^j\omega}dm}{m(B(z, e^{-u_n-x}))\int_{B(z, e^{-u_n})}h_{\sigma^j\omega}dm},
$$
where $B(z,v)$ denotes a ball of center $z$ and radius $v.$  In the limit of large $n$ the ratio on the right hand side of the preceding equality goes to $1$ by Lebesgue's differentiation theorem, while the first ratio on the left hand side goes to $e^{-x}.$  This gives the desired result with the probability distribution $\pi=1-e^{-x}.$ By generalizing we easily get the equivalent of Theorem \ref{T3} in our case
\begin{proposition}
For the random fibred system constructed above and having chosen the observable $\phi(x)=g(\text{dist}(x,z)),$ where $g$ has one of the three forms given in the statement of Theorem \ref{T3} and $z$ is chosen $m$-almost everywhere, the POT and AOT MREPP $a_nA_n$ both converge in distribution to a compound Poisson distribution process with intensity $\theta=1$ and multiplicity distribution 
\begin{equation}
\pi(x)=\begin{cases}
1-\e^{-x}, \text{when $g$ is of type 1 and $a_n=h(u_n)^{-1}$}\\
1-(1+x)^{-\beta}, \text{when $g$ is of type 2 and $a_n=u_n^{-1}$}\\
1-(1-x)^{\gamma}, \text{when $g$ is of type 3 and $a_n=(D-u_n)^{-1}$}
\end{cases}
\end{equation}
\end{proposition}

\bibliographystyle{amsalpha}
\bibliography{/Users/jmfreita/GoogleDriveNoSpace/Bibliografia/Bibliografia}


\end{document}